\documentclass[twoside]{article}

\usepackage{authblk}
\usepackage{mathrsfs}
\usepackage{amsmath}     				
\usepackage{amssymb}   					
\usepackage{amsthm}     				
\usepackage{graphicx}    				

\usepackage{textcomp}    				
\usepackage[T1]{fontenc} 				
\usepackage{marvosym}    				
\usepackage[sc]{mathpazo}  	 			

\usepackage{authblk} 					
\usepackage[usenames]{xcolor}  			
\usepackage{lastpage} 					

\usepackage{enumerate}	 				

\definecolor{ForestGreen}{rgb}{0.15,0.416,0.18}
\definecolor{EgyptBlue}{rgb}{0.063,0.2,0.65}
\usepackage{hyperref} 					

\hypersetup{colorlinks, allcolors=black}

\definecolor{ForestGreen}{rgb}{0.15,0.416,0.18}
\definecolor{EgyptBlue}{rgb}{0.063,0.2,0.65}
\hypersetup{
	colorlinks=true,
	linkcolor=EgyptBlue,
	citecolor=EgyptBlue,
	urlcolor=ForestGreen
}

\newtheorem{theorem}{Theorem}[section]

\newtheorem{lemma}[theorem]{Lemma}
\newtheorem{proposition}[theorem]{Proposition}

\theoremstyle{definition}

\theoremstyle{definition}
\newtheorem{remark}[theorem]{Remark}
\theoremstyle{definition}

\numberwithin{equation}{section}
\numberwithin{table}{section}
\numberwithin{figure}{section}

\title{The uniqueness and concentration behavior of solutions for  a  nonlinear fractional Schr\"odinger system}

\author[a]{\textbf{Chungen Liu\footnote{Email addresse: liucg@nankai.edu.cn}}}
\author[a]{\textbf{Zhigao Zhang\footnote{Email addresse: zhangzhigao@e.gzhu.edu.cn}}}
\author[a]{\textbf{Jiabin Zuo\footnote{Corresponding author. Email \textcolor{red}{address}: E-mail:zuojiabin88@163.com}}}

\affil[a]{School of Mathematics and Information Science, Guangzhou University, Guangzhou, 510006, PR China}

\makeatletter
\renewcommand{\maketitle}{\bgroup\setlength{\parindent}{0pt}

	\vspace{1truecm}
	\begin{center}{\vbox{\titlefont\@title}}\end{center}
	\vspace{0.5truecm}
	\begin{center}{\@author} \end{center}
	
	\egroup
}

\renewcommand{\@fnsymbol}[1]{%
	\ifcase#1 \or {\,\Letter\!} \or\textasteriskcentered\or \textasteriskcentered\textasteriskcentered
	\else\@ctrerr\fi}
\makeatother

\newcommand*{\titlefont}{\fontsize{18}{21.6}\selectfont\textbf}

\makeatletter
\renewcommand\@author{\ifx\AB@affillist\AB@empty\AB@author\else
	\ifnum\value{affil}>\value{Maxaffil}\def\rlap##1{##1}%
	\AB@authlist\\[\affilsep]\vbox{\AB@affillist}
	\else  \AB@authors\fi\fi}
\makeatother


\begin{document}
	
	\maketitle
	
	\pagestyle{plain}

	\begin{center}
		\noindent
		\begin{minipage}{0.85\textwidth}\parindent=15.5pt
			
			\bigskip
			
			{\small{
					\noindent {\bf Abstract.} The paper is concerned with a nonlinear system of two coupled fractional Schr\"odinger equations  with both attractive intraspecies and attractive interspecies interactions in $\mathbb{R}$. By analyzing  an associated $L^2$-constrained minimization problem, the  uniqueness of solutions to this system is proved via the implicit
					function theorem.  Under a certain type of trapping potential, by establishing some delicate
					energy estimates, we present  a detailed analysis on the  concentration behavior of the solutions as the total 
					strength of intraspecies and interspecies  interactions tends to a critical value, where each component of the solutions blows up and concentrates at a flattest common minimum point  of the
					associated trapping potentials.  An optimal
					blow-up rate of solutions to the system is also given.}
				\smallskip
				
				\smallskip
				
			}
			
		\end{minipage}
	\end{center}
	\section{Introduction}\label{Intro}
	\hspace{1.5em}In this  paper, we consider the following nonlinear system of two coupled fractional Schrödinger equations in  $\mathbb{R}$:\\
	\begin{equation}\label{system1.1}
		\begin{cases}
			\sqrt{-\Delta }u_{1}+V_{1}(x)u_{1} =\mu_{1}u_{1}+a_{1} u_{1}^{3}+\beta u_{2}^{2}u_ {1},\\
			\sqrt{-\Delta }u_{2}+V_{2}(x)u_{2} =\mu_{2}u_{2}+a_{2} u_{2}^{3}+\beta u_{1}^{2}u_ {2},
		\end{cases}
	\end{equation}\\
	where for $i=1,2$, $\mu_i\in \mathbb{{R}}$ is a  suitable Lagrange multiplier which denotes the chemical potential, $V_i(x)$ is a certain type of trapping potential,  $a_i>0$ $(resp. < 0)$ corresponds to the attractive $(resp.\ repulsive)$ intraspecies interaction inside each component, and 
	 $\beta>0$ $(resp. < 0)$ represents the attractive $(resp.\ repulsive)$ interspecies interaction between the two components.

	The  system (\ref{system1.1}) arises  from the following time-dependent nonlinear fractional Schr\"{o}dinger equation
	\begin{equation*}\label{ori}
		i\partial_t \psi(x,t) = (-\Delta)^{\alpha/2} \psi + V(x)\psi + \beta |\psi|^2 \psi, \quad x \in \mathbb{R},\ t > 0,
\end{equation*}
which is introduced to model non-Gaussian Bose-Einstein condensates (BEC), where bosons follow Lévy flights instead of Gaussian paths, see \cite{T14}. Here, $\alpha$ is the Lévy index which measures self-similarity
in the Lévy fractal path,  $V(x)$ denotes the external trapping potential, $\beta $ describes the strength of local (or short-range) interactions between particles, $\psi(x,t)$ is a complex-valued wave function given by $\psi(x,t) = e^{-i\mu t} u(x)$, where $u(x)$ is a solution of  the fractional single-component equation
	
	\begin{equation}\label{ori2}
		(-\Delta)^s u + V(x)u = \mu u + \beta u^3, \quad \text{in } \mathbb{R}.
\end{equation}
   For more physical background, we refer the reader to references \cite{T9, T13} and the references therein.

   In recent years, under various conditions
   on $V(x)$ and $\mu$, the analogous problems to (\ref{ori2}) were  widely investigated.  For the case $V(x) = 0$, the existence of solutions was considered in \cite{T11, T8,  T15} for different growth conditions of nonlinearities; while for $V(x)\neq 0$, one has to take into account the effect
   of the potential $V(x)$ besides the nonlinearity. As shown in \cite{T21, T25, T20}, which focused on fractional Schr\"odinger equations with various potentials, the existence and non-existence of $L^2$-normalized solutions were established, and a detailed analysis of the concentration behavior of solutions was also presented. Moreover, the frequency $\mu$ is also a critical parameter, and there exist two substantially distinct approaches to it.   One is to regard it as fixed, as  studied in \cite{ T4,  T12}; the alternative is to treat it as unknown, where it appears as a Lagrange multiplier associated with the prescribed mass constraint. The latter approach  is particularly interesting from a physical perspective, since the mass often has a clear physical meaning, with relevant research  presented  in \cite{T1,  T6, T10, T5}.
   
   Note that within the multiple-component framework, interspecies interactions between the two components give rise to  more elaborate physical phenomena and  render the corresponding analytic investigations more challenging.  When (\ref{system1.1}) reduces to a Laplacian system, that is,	\begin{equation}\label{system2.1}
   	\begin{cases}
   		-\Delta u_{1}+V_{1}(x)u_{1} =\mu_{1}u_{1}+b_{1} u_{1}^{3}+\beta u_{2}^{2}u_ {1},\\
   		-\Delta u_{2}+V_{2}(x)u_{2} =\mu_{2}u_{2}+b_{2} u_{2}^{3}+\beta u_{1}^{2}u_ {2},
   	\end{cases}
   \end{equation}\\
     it becomes a classical model in $\mathbb{{R}}^2$ describing two-component   Bose-Einstein condensates, which was studied  in \cite{T18, T28} for the case where the  intraspecies and  interspecies interactions  are both
     attractive (i.e., $b_1,b_2,\beta>0$).  Under the constraint condition \begin{equation*}\label{M1}
   	 \left \{ (u_1,u_2)\in \tilde{\mathcal{H}}_1\times\tilde{\mathcal{H}}_2:\int_{\mathbb{R}^2 }\left| u_1\right|^{2}dx = \int_{\mathbb{R}^2 }\left|u_2\right|^{2}dx =1\right \},
   \end{equation*}
    where $\tilde{\mathcal{H}}_i=\{ u\in H^{1}(\mathbb{R}^2): \int_{\mathbb{R}^2 }V_{i}(x)\left|u\right|^2dx<\infty \}$ $(i=1,2)$, in \cite{T18}, Guo, Zeng and  Zhou  analyzed how the existence, non-existence and uniqueness of solutions to (\ref{system2.1}) depend on the  parameters $b_1$, $b_2$ and $\beta$, and   further showed that, as the total 
    strength of intraspecies and interspecies  interactions tends to a critical value, each component of the solutions blows up at the same point, and under  a specific scaling transformation, converges to the same limit function in $H^{1}(\mathbb{R}^2)$. Separately, in \cite{T28},  under a different constraint condition  
   $ 	\{ (u_1,u_2)\in \tilde{\mathcal{H}}_1\times\tilde{\mathcal{H}}_2:\int_{\mathbb{R}^2 }( \left|u_1\right|^{2} 
   + \left|u_2\right|^{2}) dx =1\}$, the  Laplacian system (\ref{system2.1}) with $\mu_{1}=\mu_{2}$ was investigated.  Not only was a   
   complete classification of the existence and non-existence of ground states  provided, but the  uniqueness, mass concentration, and symmetry-breaking of ground states  were also  established under different types of trapping potentials as the  interspecies interaction strength approached  a critical value. In addition, Guo, Zeng and  Zhou  generalized the results of \cite{T18} to system (\ref{system2.1}) with attractive intraspecies and repulsive interspecies interactions $(\beta<0)$ in \cite{T19}, and notably, the asymptotic analysis differs distinctly from that in \cite{T18, T28} due to the opposite signs of these two types of interactions. It can be  seen that the uniqueness and   decay rate of the positive solution to the associated classical equation, together with the strong maximum principle, play a key role in the concentration behavior of solutions  in \cite{T18, T19, T28}. 
   
\textcolor{red}{   We also note that in recent work \cite{T new},  the authors decompose the energy functional and, via refined estimates, prove the Palais-Smale condition, thereby obtaining existence, multiplicity, and asymptotic results of solutions for double critical fractional Schr\"odinger-Poisson systems. The simultaneous appearance of double critical exponents and non-local terms requires control of the relationship between the two critical exponents, which differs from the setting in systems (\ref{system1.1}) and (\ref{system2.1}).}
   
    Motivated by \cite{T18, T19}, Liu, Zhang and Zuo  considered the fractional Laplacian system (\ref{system1.1})  in \cite{T27}, where various existence and non-existence results were obtained under appropriate ranges for the parameters $a_1>0$, $a_2>0$, and $\beta\in \mathbb{{R}}$. However,  the  uniqueness and  concentration behavior of solutions  remain unexplored. To address this gap,  we  focus  on these properties of  solutions to the system 
   (\ref{system1.1}) with attractive   intraspecies and interspecies interactions, i.e.,  $a_1>0$, $a_2>0$, and $\beta>0$. The case of $\beta<0$ is left to interested readers for future investigation.  
   
   In this paper,    trapping potentials $V_i(x)$ satisfy the following condition\\

   $(\mathcal{D})$ $V_{i}(x)\in L^{\infty } _{loc}(\mathbb{R}) $,  $\displaystyle\lim_{\lvert  x\rvert \rightarrow \infty }V_{i}(x)=\infty $  and  $\mathop{\mathrm{inf\mathrm{} } }\limits _{x\in \mathbb{R} }V_{i}(x)=0$,\ \  $i=1,2$.
   
  \noindent 
   Moreover, we assume throughout the paper that  both  $\mathop{\mathrm{inf\mathrm{} } } _{x\in \mathbb{R} }\big(V_{1}(x)+V_{2}(x)\big)$  and  $\mathop{\mathrm{inf\mathrm{} } } _{x\in \mathbb{R} }V_{i}(x) \ (i=1,2)$ are  attained.
     
   Before  formulating the main results of this paper, we first give some notations and recall some known results. Let the $\frac{1}{2}$-Laplace operator be defined in the Sobolev space $H^\frac{1}{2}(\mathbb{R})$ by the
   Fourier transform, i.e.,
   \begin{equation*}\label{four}
   	\sqrt{-\Delta  } u:=\mathscr{F}^{-1}(\left | \xi  \right | \mathscr{F}u(\xi)),\ \ u\in H^{\frac{1}{2} }(\mathbb{R}),
   \end{equation*}
   in which $	\mathscr{F}u(\xi )$ represents the Fourier transform of $u$, i.e.,
   \begin{equation*}
   	\mathscr{F}u(\xi )=\frac{1}{\sqrt{2\pi}}\int_{\mathbb{R}}\exp^{- i\xi \cdot x}u(x)dx,\ \ u\in H^{\frac{1}{2} }(\mathbb{R}),
   \end{equation*}
   and $H^{\frac{1}{2} }(\mathbb{R})$ is defined as 
   \begin{equation*}
   	H^{\frac{1}{2} }(\mathbb{R}):=\{ u\in L^{2} (\mathbb{R} ):\left \| u\right \|_{ H^{\frac{1}{2} }(\mathbb{R})}^{2}:=\int_{\mathbb{R}}(1+2\pi\left | \xi  \right | ) \Big| \mathscr{F}u(\xi)  \Big|^{2}d\xi< \infty   \},
   \end{equation*}
   with  the norm  $\left \| \cdot  \right \| _{H^{\frac{1}{2}}(\mathbb{{R}})} $    given by
   \begin{equation*}
   	\left \| u  \right \| _{H^{\frac{1}{2}}(\mathbb{{R}})}^2 :=\int_{\mathbb{R}}\Big[u(-\Delta )^{\frac{1}{2} } u+u^{2}\Big]dx=\int_{\mathbb{R}}\Big[|(-\Delta )^{\frac{1}{4} } u|^{2}+u^{2}\Big]dx.
    \end{equation*}

It is known that (\ref{system1.1}) is the system of Euler-Lagrange equations for the following constrained
minimization problem
\begin{equation}\label{problem 1.7}	
	\hat{e} (a_{1}, a_{2}, \beta):=\mathop{\mathrm{inf\mathrm{} } }\limits _{(u_{1}, u_{2})\in \mathcal{M}}E_{a_{1}, a_{2} , \beta}(u_{1}, u_{2}),
\end{equation}
where the energy functional $	E_{a_{1}, a_{2} , \beta}(u_{1}, u_{2})$ is defined by
\begin{align}\label{eq1.9}
	E_{a_{1}, a_{2} , \beta}(u_{1}, u_{2})=&\sum_{i=1}^{2} \int_{\mathbb{R}} [u_{i}\sqrt{-\Delta  } u_{i}+V_{i}(x)u_{i}^{2}]dx - \sum_{i=1}^{2}\frac{a_{i}}{2}\int_{\mathbb{R}}\left | u_{i} \right |^{4}dx \nonumber  \\
	&-\beta\int_{\mathbb{R}}  \left | u_{1} \right |^{2}\left | u_{2} \right |^{2} dx,\quad \text{for any } (u_1, u_2)\in \mathcal{X}.
\end{align}
Here,  $\mathcal{M}$ is the so-called mass constraint   
\begin{equation}\label{M}
	\mathcal{M} =\left \{ (u,v)\in \mathcal{X}:\int_{\mathbb{R} }u^{2}dx = \int_{\mathbb{R} }v^{2}dx =1\right \}
\end{equation}
with
\begin{equation*}
	\mathcal{X}=\mathcal{H}_{1}\times\mathcal{H}_{2},\ \  \mathcal{H}_{i}=\left\{ u\in H^{\frac{1}{2} }(\mathbb{R}): \int_{\mathbb{R} }V_{i}(x)u^{2}dx<\infty \right\}\ \ (i=1,2) \quad
\end{equation*}
and 
\begin{equation*}\label{eq1.8}
	\left \| u  \right \| _{\mathcal{H}_{i} }^{2}=\int_{\mathbb{R} }\left [  | (-\Delta  ) ^{\frac{1}{4} }u|^{2}+V_{i}(x)u^{2} \right ]dx, \ \  \left \| u \right \|_{\mathcal{X}}^{2}=\left \| u \right \|_{\mathcal{H}_{1}}^{2}+\left \| u \right \|_{\mathcal{H}_{2}}^{2},\quad i=1,2.
\end{equation*}
The classical equation associated to system (\ref{system1.1}) is known as 
	\begin{equation}\label{classical eq}
	\sqrt{-\Delta }u+u=u^{3},\ \ \:u\in H^{\frac{1}{2}}(\mathbb{{R}}).
\end{equation}
Let $Q(x)>0$ be the  unique radially symmetric ground state (\cite{T11,T26}) of (\ref{classical eq}), and it is shown in Proposition 3.1 of \cite{T26} that the ground state $Q$ has the decay rate
\begin{equation}\label{decay}
	Q(x)=O(|x|^{-2})\ \ \text{as}\ \  |x|\to \infty.
\end{equation}
 Denoting $a^{*}:=\left \| Q \right \|_{2}^{2}$, it is  known from Theorem 1.1 of \cite{T25} that the fractional Gagliardo-Nirenberg inequality corresponding to (\ref{classical eq}) takes the following form
 	\begin{equation}\label{eq1.10}
 	\int_{\mathbb{R}}\left | u \right |^{4}dx\le \frac{2}{a^*} \int_{\mathbb{R}}| \left ( -\Delta   \right )^{\frac{1}{4}}u   |^{2}dx    \int_{\mathbb{R}}\left | u \right |^{2}dx,\ \ \:u\in H^{\frac{1}{2}}(\mathbb{{R}}).
 \end{equation}
 Furthermore, by Lemma 5.4 of \cite{T11} and Remark 2.4 of \cite{T21}, $Q(x)$ satisfies the following Pohozaev identity
\begin{equation}\label{eq1.11}
	\| \left (-\Delta   \right )^{\frac{1}{4} } Q \|_{2}^{2}= \left \| Q \right \|_{2}^{2}=\frac{1}{2}\int_{\mathbb{R}} |Q|^{4}dx  .
\end{equation}
Based  on the above notations, we know from Theorem 1.1 in \cite{T27} that when  $(\mathcal{D})$  holds,  (\ref{problem 1.7}) has a minimizer for any $0<a_1, a_2<a^*$, and  $0<\beta<\sqrt{(a^*-a_1)(a^*-a_2)}$, but (\ref{problem 1.7})  has no  minimizer if $a_1>a^*$, or $a_2>a^*$, or $\beta>\frac{a^*-a_1}{2}+\frac{a^*-a_2}{2}$.
 
 We now present the main results of this paper. Our first main result is concerned with 
 the uniqueness of minimizers for problem (\ref{problem 1.7}).
 \begin{theorem}\label{small}
 	If $V_i(x)$ satisfies $(\mathcal{D})$  for $i=1,2$, then (\ref{problem 1.7}) has a unique
 	non-negative minimizer if $|(a_1, a_2, \beta)|$ is suitably small.
 \end{theorem}
 
 \begin{remark}
 	Some results on uniqueness for the single-component equations analogous to (\ref{ori2}) were established by using the contraction mapping approach in \cite{T22, T23} and the implicit function theorem in \cite{T25}. In proving  Theorem \ref{small},  we adopt strategies similar to those in \cite{T25}.
 \end{remark}

 For      any fixed $0 < \beta < a^*$, we investigate the limit behavior of the minimizers for  (\ref{problem 1.7}) as $(a_1, a_2)$ approach the threshold $(a^* - \beta, a^* - \beta)$ in the case where  (\ref{problem 1.7})   has no minimizer at the threshold. According to  Theorem 1.6 (i) of \cite{T27}, we will consider the special case where  $\mathop{\mathrm{inf\mathrm{} } } _{x\in \mathbb{R} }\big(V_{1}(x)+V_{2}(x)\big)=0$, which means that the minima of $V_1(x)$ coincide with those of $V_2(x)$. Moreover, for $i=1,2$, we assume the trapping potential $V_{i}(x)$  has a finite number of isolated minima,   and  in their vicinity, it behaves like a power of the distance from these points, that is, $V_{i}(x)$ takes the form of
 \begin{equation}
 \begin{aligned}\label{form 1}&V_{i}(x)=h_{i}(x)\prod_{j=1}^{n_{i}}|x-x_{ij}|^{p_{ij}}\quad\mathrm{with}\quad C<h_{i}(x)<1/C\quad\mathrm{in}\ \ \mathbb{R},\\&\text{and}\ \ h_{i}(x)\in C_{\mathrm{loc}}^{\alpha}(\mathbb{R})\quad\mathrm{for~some}\quad \alpha\in(0,1),\end{aligned}
 \end{equation}
 where $n_i\in \mathbb{N}$, $0<p_{ij}<3$, $x_{ij}\neq x_{ik}$ for $j\neq k$, and $\lim_{x\to x_{ij}}h_{i}(x)$ exists for all $1\leq j\leq n_{i}$. In addition, we suppose, \textcolor{red}{without loss of generality}, that there exists an integer $l$ satisfying $1 \leq l \leq \min\{n_1, n_2\}$ such that
 \begin{equation}\label{point}
 	\begin{aligned}&x_{1j}=x_{2j},\quad\mathrm{where}\quad j=1,\cdots,l;\\&x_{1j_{1}}\neq x_{2j_{2}},\quad\mathrm{where}\quad j_{i}\in\begin{Bmatrix}l+1,\cdots n_{i}\end{Bmatrix}\ \ \mathrm{and}\ \ i=1,2.\end{aligned}
 \end{equation}
 Notice that, (\ref{point}) indicates
 \begin{equation}\label{Lamda}
 	\Lambda:=\left\{x\in\mathbb{R}:V_1(x)=V_2(x)=0\right\}=\{x_{11}, x_{12},\cdots, x_{1l}\}.
 \end{equation}
 Define
 \begin{equation}\label{index}
 	\bar{p}_{j}:=\min\left\{p_{1j}, p_{2j}\right\},\ j=1, \cdots, l;\quad p_{0}:=\max_{1\leq j\leq l}\min\left\{p_{1j}, p_{2j}\right\}=\max_{1\leq j\leq l}\bar{p}_{j}, 
 \end{equation}
 so that
 \begin{equation}\label{1.17}
 	\bar{\Lambda}:=\begin{Bmatrix}x_{1j}: \bar{p}_j=p_0,\ j=1,\cdots,l\end{Bmatrix}\subset\Lambda.
 \end{equation}
 Additionally for $j=1,\cdots, l$, we define $\gamma_j \in(0,\infty]$ by 
 \begin{equation}\label{gamma}
 	\gamma_j:=\lim_{x\to x_{1j}}\frac{V_1(x)+V_2(x)}{|x-x_{1j}|^{p_0}}.
 \end{equation}
 Note that for $j=1,\cdots, l$,  $\gamma_{j}<\infty$ if and only if $x_{1j}\in\bar{\Lambda}$. Finally, we introduce $\gamma = \min \left\{ \gamma_1, \dots, \gamma_l \right\}$ so that the set
 \begin{equation}\label{flattest}
 	\mathcal{Z} := \left\{ x_{1j} : \gamma_j = \gamma, \ 1 \leq j \leq l \right\} \subset \bar{\Lambda}
 \end{equation}
 represents the locations of the flattest global minima  of $V_1(x) + V_2(x)$. With the above notations, the second significant result of the present paper can be formulated as follows.
 
 \begin{theorem}\label{blow}
 	Let $0< \beta < a^*$, and  suppose that for $i=1,2$,    $p_{ij}\in(0,1)$  for  $1\leq j\leq n_{i}$   and  $V_i( x)$ satisfies (\ref{form 1}) - (\ref{point}). Let $(u_{a_1},u_{a_2})$ be a non-negative minimizer of (\ref{problem 1.7}) as $(a_1,a_2) \nearrow( a^* - \beta , a^* - \beta )$.
 	Then for any  sequence $\{(a_{1k}, a_{2k})\}$ satisfying  $(a_{1k}, a_{2k})\nearrow ( a^* - \beta , a^* - \beta )$ as $k\to\infty$, there exists a subsequence of  $\{(a_{1k}, a_{2k})\}$,  still denoted by $\{(a_{1k}, a_{2k})\}$, such that for
    $i = 1, 2$, each $u_{a_{ik}}$ has at least one  global maximum point $x_{ik}\xrightarrow{k}\bar{x}_0$ for some  $\bar{x}_0\in\mathcal{Z}$ and 
    \begin{equation*}
    	\lim_{k\to\infty}\frac{|x_{ik}-\bar{x}_0|}{(a^*-\frac{a_{1k}+a_{2k}+2\beta}{2})^{\frac{1}{p_0+1}}}=0.
    \end{equation*}
     In addition, for $i=1,2$,
     \begin{equation*}
     	\lim_{k\to\infty}\left(a^*-\frac{a_{1k}+a_{2k}+2\beta}{2}\right)^{\frac{1}{2(p_0+1)}}u_{a_{ik}}\left(\left(a^*-\frac{a_{1k}+a_{2k}+2\beta}{2}\right)^{\frac{1}{p_0+1}}x+x_{ik}\right)=\frac{\lambda^{\frac{1}{2}}}{\|Q\|_2}Q(\lambda x)
     \end{equation*}
     strongly in $H^{\frac{1}{2}}(\mathbb{{R}})$. Here  $\lambda$ is defined as 
     \begin{equation*}
     	\lambda=\left(\frac{p_0\gamma}{2}\int_{\mathbb{R}}|x|^{p_0}Q^2(x)dx\right)^{\frac{1}{p_0+1}},
     \end{equation*}
     where $p_0>0$ and $\gamma>0$ are defined in (\ref{index}) and (\ref{flattest}), respectively.
 \end{theorem}
 \begin{remark}
 	 The existence of the global maximum point for each non-negative minimizer of problem (\ref{problem 1.7}) relies on  the non-local De Giorgi-Nash-Moser theory introduced in \cite{T16} and \cite{T17}. However, the uniqueness of the global maximum point for  problem (\ref{problem 1.7}) remains unproven due to the deficiency of  elliptic regularity theory in the fractional Sobolev space, which is essential for proving the smoothness of the minimizer sequence. Consequently, the question of uniqueness remains open for future investigation.
 \end{remark}
 
 \begin{remark}
 	The uniqueness and concentration behavior of minimizers obtained in this paper depend on the existence and nonexistence results given in Theorems 1.1 and 1.6 of \cite{T27}, which left the case $s\neq \tfrac12$ open; consequently, the same case remains open in our work.
 \end{remark}
 
 Compared with the Laplacian system in \cite{T18},  some new difficulties arise due to the presence of the non-local operator $\sqrt{-\Delta }$. Firstly, the strong maximum principle as in  \cite{T18} does not work to establish the positivity of the solution to the fractional Laplacian equation (\ref{3.48}). To resolve this difficulty, we deduce the positivity of the solution from the strict positivity of the kernel $G_{1/2,\lambda}(x-y)$ of the resolvent $(\sqrt{-\Delta}+\lambda)^{-1}$ on $\mathbb{R}$, as established in Lemma C.1 of \cite{T26}.
 Secondly, due to the non-uniqueness of positive solutions to (\ref{classical eq}), we prove that under a suitable translation and a dilation, each component of the non-negative minimizer $(u_{d_1}, u_{d_2})$ of (\ref{new problem}) converges to the ground state of (\ref{3.50}). Then, via a coordinate transformation, we establish the equivalence between  this ground state and that of (\ref{classical eq}), thereby overcoming the obstacle arising from the non-uniqueness of positive solutions  by exploiting the uniqueness of the ground state of (\ref{classical eq}).
  Lastly, in striking contrast to the classical Schrödinger equation with the Laplacian in which the ground state decays exponentially at infinity, the ground state of (\ref{classical eq}) decays only polynomially at infinity. This is why we first require that the order $p_{ij}<3$ for the polynomial potential function (\ref{form 1}), which ensures that $\left|\cdot\right|^{{p}_{0}}Q^2(x)$ is integrable. Moreover, to derive the optimal energy estimates for $e_{i}(d_i)$$(i=1,2)$, a crucial  step that contributes to $L^{4}(\mathbb{{R}})$ - estimates of minimizers, we further impose that $p_{ij}<1$ (see the forthcoming estimate (\ref{limitiation})). Consequently,  we assume that $0 < p_{ij} < 1$ for $i = 1, 2$, $j=1, \cdots, n_i$.
 
    From Theorems \ref{small} and \ref{blow}, we find that symmetry breaking occurs in the minimizers of $\hat{e} (a_{1}, a_{2}, \beta)$  when the potentials $V_1(x)$ and $V_2(x)$ possess symmetry. For example, take the trapping potentials $V_1(x)$ and $V_2(x)$ of the form
 \begin{equation*}
 	V_1(x)=V_2(x)=\prod_{j=1}^l|x-x_j|^p,\ p>0,
 \end{equation*} 
 where the points $x_j$ with $j=1, \cdots, l$ are  arranged on the vertices of a regular polyhedron centred at the origin. It then follows from Theorem \ref{small} that there exist ${a}_*$ and $\bar{a}$ satisfying   $0<{a}_*\leq \bar{a}<a^*$, such that for $0<a_i+\beta<a_*(i=1,2)$, $\hat{e} (a_{1}, a_{2}, \beta)$ admits a unique non-negative minimizer with the same symmetry as $V_1(x)=V_2(x)$. However, symmetry breaking occurs when $\bar{a}<a_i+\beta<a^*(i=1,2)$, since Theorem \ref{blow} implies that  $\hat{e} (a_{1}, a_{2}, \beta)$ admits (at least) $l$ different non-negative minimizers, with both of their components concentrating at a zero point of $V_1(x)=V_2(x)$.

This paper is organized as follows. Section \ref{section existence}  focuses on the proof of Theorem \ref{small} on the uniqueness of non-negative minimizers for $\hat{e} (a_{1}, a_{2}, \beta)$ when  $|(a_1, a_2, \beta)|$ is suitably small. In section \ref{section2}, we prove Proposition \ref{pro 3.1}  on the optimal energy estimates of minimizers, upon which the proof of Theorem \ref{blow} is addressed in Section \ref{section new}. 
 
 \section{Uniqueness of non-negative minimizers} \label{section existence}
 
 In this section, we will prove Theorem \ref{small}   by using the
 implicit function theorem. To this end, we start with the following compactness result.
 
	\begin{lemma}\label{lemma2.1}
	(\cite[Lemma 3.1]{T25}) Suppose  \  $0\le V_{i}(x)\in L^{\infty } _{loc}(\mathbb{R}) $ \ satisfies  $\displaystyle\lim_{\lvert  x\rvert \rightarrow \infty }V_{i}(x)=\infty $,   $i=1,2$. Then for all $2\le q < \infty$, the embedding $\mathcal{H}_{1}\times\mathcal{H}_{2}\hookrightarrow L^{q}(\mathbb{R} )\times L^{q}(\mathbb{R} )$ is compact.
\end{lemma}

Define
\begin{equation}\label{lambda}
	\lambda_{i1}=\inf\left\{\int_{\mathbb{R}}\left( | (-\Delta  ) ^{\frac{1}{4} }u|^{2}+V_i(x)u^2\right)dx:u\in\mathcal{H}_i\ \ \mathrm{and}\ \ \ \int_{\mathbb{R}}|u|^2dx=1\right\}, \quad i=1,2.
\end{equation}
Usually, for $i=1,2$, $\lambda_{i1}$ denotes the first eigenvalue of $\sqrt{-\Delta}+V_i(x)$ in $\mathcal{H}_i$. Moreover, it follows from  Lemma \ref{lemma2.1} that $\lambda_{i1}$ is simple and can
be attained by a positive normalized function $\Psi_{i1}\in\mathcal{H}_i$, which is called the first eigenfunction of $\sqrt{-\Delta}+V_i(x)$, $i=1,2$. Let
\begin{equation*}
	\mathcal{L}_i=\operatorname{span}\{\Psi_{i1}\}^{\perp}=\left\{u\in\mathcal{H}_{i}:\int_{\mathbb{R}}u\Psi_{i1}dx=0\right\},\quad i=1,2,
\end{equation*}
so that
 \begin{equation*}
 	\mathcal{H}_i=\mathrm{span}\{\Psi_{i1}\}\oplus \mathcal{L}_i,\quad i=1,2.
 \end{equation*}
 Then, we have the following proposition.
 \begin{proposition}(\cite[Lemma A.1]{T25})\label{pu}
 	Suppose $V_i(x)$ satisfies $(\mathcal{D})$ for i = 1, 2, then we have
\begin{flalign*}
	&(\mathrm{i})\ \ \mathrm{ker}\bigl(\sqrt{-\Delta}+V_{i}(x)-\lambda_{i1}\bigl) = \mathrm{span}\{\Psi_{i1}\}, &\\
	&(\mathrm{ii})\ \Psi_{i1} \notin \bigl(\sqrt{-\Delta}+V_{i}(x)-\lambda_{i1}\bigl)\mathcal{L}_i, & \\
	&(\mathrm{iii})\ \mathrm{Im}\bigl(\sqrt{-\Delta}+V_{i}(x)-\lambda_{i1}\bigl) = \bigl(\sqrt{-\Delta}+V_{i}(x)-\lambda_{i1}\bigl) \mathcal{L}_i\ \text{is closed in}\ \mathcal{H}_i^*, & \\
	&(\mathrm{iv})\ \mathrm{codim}\ \mathrm{Im}\bigl(\sqrt{-\Delta}+V_{i}(x)-\lambda_{i1}\bigl) = 1, &
\end{flalign*}
\noindent where $\mathcal{H}_i^*$ denotes the dual space of $\mathcal{H}_i$, $i=1,2$.
 \end{proposition}
 
 Consider  the  $C^1$ functional  $G_i:\mathcal{X}\times\mathbb{R}^3\mapsto\mathcal{H}_i^*$ defined by
 \begin{equation}\label{Gi}
 	G_i(u_1,u_2,\mu_i,a_i,\beta)=\bigl(\sqrt{-\Delta}+V_i(x)-\mu_i\bigr)u_i-a_iu_i^3-\beta u_j^2u_i,\quad i=1,2,
 \end{equation}
 where $j \neq i$ and $j = 1, 2$. Inspired by (\cite[Lemma A.2]{T25}) and (\cite[Lemma 4.1]{T18}), we then present the following lemma.
 
 \begin{lemma}\label{cunzai}
 	Let $G_i$ $(i=1,2)$ be defined by (\ref{Gi}). Then, for $i=1,2$, there exist $\delta>0$ and a unique function $(u_{i}(a_{1},a_{2},\beta),\mu_{i}(a_{1},a_{2},\beta))\in C^{1}\big(B_{\delta}(\vec{0});B_{\delta}(\Psi_{i1},\lambda_{i1})\big)$ such that
 \end{lemma}
 \begin{equation}\label{2.1}
 	 \begin{cases}
 		\mu_i(\vec{0})=\lambda_{i1},\ u_i(\vec{0})=\Psi_{i1},\ i=1,2;\\G_i\left(u_1(a_1,a_2,\beta),u_2(a_1,a_2,\beta),\mu_i(a_1,a_2,\beta),a_i,\beta\right)=0,\ i=1,2;\\\|u_1(a_1,a_2,\beta)\|_2^2=\|u_2(a_1,a_2,\beta)\|_2^2=1.
 	\end{cases}
 \end{equation}

   \begin{proof}
   	Define $t_i:(\mathcal{L}_1\times\mathbb{R})\times(\mathcal{L}_2\times\mathbb{R})\times\mathbb{R}^4\mapsto\mathcal{H}_i^*$  by 
   	\begin{equation*}
   		t_i\left((\ell_1,\mu_1),(\ell_2,\mu_2),s_1,s_2,a_i,\beta\right):=G_i((1+s_1)\Psi_{11}+\ell_1, (1+s_2)\Psi_{21}+\ell_2, \mu_i, a_i, \beta),\quad i=1, 2.
   	\end{equation*}
   	Then $t_i\in C^1((\mathcal{L}_1\times\mathbb{R})\times(\mathcal{L}_2\times\mathbb{R})\times\mathbb{R}^4;\mathcal{H}_i^*)$ and
\begin{align*}t_{i}((0,\lambda_{11}),(0,\lambda_{21}),\vec{0})&=G_i(\Psi_{11},\Psi_{21},\lambda_{i1},\vec{0})=0, \notag \\D_{s_i}t_i((0,\lambda_{11}),(0,\lambda_{21}),\vec{0})&=D_{u_i}G_i(\Psi_{11},\Psi_{21},\lambda_{i1},\vec{0})\Psi_{i1}\notag \\&=\bigl(\sqrt{-\Delta}+V_i(x)-\lambda_{i1}\bigr)\Psi_{i1}=0,\\D_{s_j}t_i((0,\lambda_{11}),(0,\lambda_{21}),\vec{0})&=D_{u_j}G_i(\Psi_{11},\Psi_{21},\lambda_{i1},\vec{0})\Psi_{j1}=0,\notag\end{align*}
where $j \neq i$, $i, j=1, 2$. Furthermore, we obtain that, for any $(\hat{\ell}_i,\hat{\mu}_i)\in \mathcal{L}_i\times\mathbb{R}$,
\begin{align}\label{dengshi}
	&\left\langle D_{(\ell_{i},\mu_{i})}t_{i}((0,\lambda_{11}),(0,\lambda_{21}),\vec{0}),(\hat{\ell}_{i},\hat{\mu}_{i})\right\rangle \notag \\=&D_{u_{i}}G_{i}(\Psi_{11},\Psi_{21},\lambda_{i1},\vec{0})\hat{\ell}_{i}+D_{\mu_{i}}G_{i}(\Psi_{11},\Psi_{21},\lambda_{i1},\vec{0})\hat{\mu}_{i}\\=&\bigl(\sqrt{-\Delta}+V_i(x)-\lambda_{i1}\bigr)\hat{\ell}_{i}-\hat{\mu}_{i}\Psi_{i1}\in\mathcal{H}_{i}^{*}, \notag\end{align}
and
\begin{align*}&\left\langle D_{(\ell_j,\mu_j)}t_i((0,\lambda_{11}),(0,\lambda_{21}),\vec{0}),(\hat{\ell}_j,\hat{\mu}_j)\right\rangle \notag \\=&D_{u_j}G_i(\Psi_{11},\Psi_{21},\lambda_{i1},\vec{0})\hat{\ell}_j+D_{\mu_j}G_i(\Psi_{11},\Psi_{21},\lambda_{i1},\vec{0})\hat{\mu}_j=0.\end{align*}
We then derive from  Proposition $\ref{pu}$ and ($\ref{dengshi}$) that, for $i=1, 2$,
\begin{equation}\label{isomorphism}
	D_{(\ell_i,\mu_i)}t_i((0,\lambda_{11}),(0,\lambda_{21}),\vec{0}):\mathcal{L}_i\times\mathbb{R}\mapsto\mathcal{H}_i^*\ \text{is an isomorphism. }
\end{equation}

Now, define $P:(\mathcal{L}_{1}\times\mathbb{R})\times(\mathcal{L}_{2}\times\mathbb{R})\times\mathbb{R}^{5}\mapsto\mathcal{H}_{1}^{*}\times\mathcal{H}_{2}^{*}\times\mathbb{R}^{2}$ by
\begin{equation*}
	P\left((\ell_1,\mu_1),(\ell_2,\mu_2),s_1,s_2,(a_1,a_2,\beta)\right):=\begin{pmatrix}t_1\left((\ell_1,\mu_1),(\ell_2,\mu_2),s_1,s_2,a_1,\beta\right)\\t_2\left((\ell_1,\mu_1),(\ell_2,\mu_2),s_1,s_2,a_2,\beta\right)\\\|(1+s_1)\Psi_{11}+\ell_1\|_2^2-1\\\|(1+s_2)\Psi_{21}+\ell_2\|_2^2-1\end{pmatrix},
\end{equation*}
and set $h_i(\ell_i,s_i)=\|(1+s_i)\Psi_{i1}+\ell_i\|_2^2-1$ for $i=1,2$. Then, we  obtain that
\begin{align*}&D_{((\ell_1,\mu_1),(\ell_2,\mu_2),s_1,s_2)}P((0,\lambda_{11}),(0,\lambda_{21}),0,0,\vec{0})\notag\\&\left.=\left(\begin{array}{cccc}D_{(\ell_1,\mu_1)}t_1&D_{(\ell_2,\mu_2)}t_1&D_{s_1}t_1&D_{s_2}t_1\\D_{(\ell_1,\mu_1)}t_2&D_{(\ell_2,\mu_2)}t_2&D_{s_1}t_2&D_{s_2}t_2\\D_{(\ell_1,\mu_1)}h_1&D_{(\ell_2,\mu_2)}h_1&D_{s_1}h_1&D_{s_2}h_1\\D_{(\ell_1,\mu_1)}h_2&D_{(\ell_2,\mu_2)}h_2&D_{s_1}h_2&D_{s_2}h_2\end{array}\right.\right)\\ \notag&=\begin{pmatrix}(\sqrt{-\Delta}+V_1(x)-\lambda_{11},-\Psi_{11})&0&0&0\\0&(\sqrt{-\Delta}+V_2(x)-\lambda_{21},-\Psi_{21})&0&0\\0&0&2&0\\0&0&0&2\end{pmatrix}.\end{align*}
It then follows from (\ref{isomorphism}) that
\begin{equation*}
		D_{((\ell_1,\mu_1),(\ell_2,\mu_2),s_1,s_2)}P((0,\lambda_{11}),(0,\lambda_{21}),0,0,\vec{0}): (\mathcal{L}_1\times\mathbb{R})\times(\mathcal{L}_2\times\mathbb{R})\times\mathbb{R}^5\mapsto\mathcal{H}_1^*\times\mathcal{H}_2^*\times\mathbb{R}^2
\end{equation*}
is an isomorphism. Hence, by applying the implicit function theorem, we deduce that there exist $\delta>0$ and a unique function
$(\ell_i(a_1,a_2,\beta),\mu_i(a_1,a_2,\beta),s_i(a_1,a_2,\beta))\in C^{1}(B_{\delta}(\vec{0});B_{\delta}(0,\lambda_{i1},0)),$ where $i=1, 2$, such that
\begin{equation}\label{2.10}
	\begin{cases}P\left((\ell_1,\mu_1),(\ell_2,\mu_2),s_1,s_2,(a_1,a_2,\beta)\right)=P((0,\lambda_{11}),(0,\lambda_{21}),0,0,\vec{0})=\vec{0},\\\ell_i(\vec{0})=0,\ \mu_i(\vec{0})=\lambda_{i1},\ s_i(\vec{0})=0,\ i=1,2.&\end{cases}
\end{equation}
 Set 
 \begin{equation*}
 	u_i(a_1,a_2,\beta)=(1+s_i(a_1,a_2,\beta))\Psi_{i1}+\ell_i(a_1, a_2,\beta),\ (a_1,a_2,\beta)\in B_\delta(\vec{0}),\quad i=1,2,
 \end{equation*}
 which  together with (\ref{2.10}) implies that there exists a unique function
\begin{equation*}
	( 	u_i(a_1, a_2, \beta),\mu_i(a_1,a_2,\beta))\in C^1(B_\delta(\vec{0});B_\delta(\Psi_{i1},\lambda_{i1})), \quad i=1, 2,
\end{equation*}
 such that
\begin{equation*}
	u_i(\vec{0})=(1+s_i(\vec{0}))\Psi_{i1}+\ell_i(\vec{0})=\Psi_{i1},\ \mu_i(\vec{0})=\lambda_{i1},
\end{equation*}
and
\begin{equation*}
	\begin{pmatrix}G_1\left(u_1(a_1,a_2,\beta),u_2(a_1,a_2,\beta),\mu_1(a_1,a_2,\beta),a_1,\beta\right)\\G_2\left(u_1(a_1,a_2,\beta),u_2(a_1,a_2,\beta),\mu_2(a_1,a_2,\beta),a_2,\beta\right)\\\|u_1(a_1,a_2,\beta)\|_2^2-1\\\|u_2(a_1,a_2,\beta)\|_2^2-1\end{pmatrix}=\vec{0}.
\end{equation*}
Hence, $(\ref{2.1})$ holds and the proof of Lemma \ref{cunzai} is completed.
    \end{proof}
Based on  Lemma \ref{cunzai}, we next show  the uniqueness of non-negative minimizers of (\ref{problem 1.7}) for suitably  small $|(a_1,a_2,\beta)|$.

\begin{proof}
As stated in Theorem 1.1 of \cite{T27}, when $V_1(x)$ and $V_2(x)$ satisfy  $(\mathcal{D})$,  $\hat{e}(a_1,a_2,\beta)$ has a minimizer if $0<a_1,a_{2}<a^*$ and  $\beta<\sqrt{(a^{*}-a_{1})(a^{*}-a_{2})}$. We first claim that $\hat{e}(a_1, a_2, \beta)$ is continuous on the region $H:=(-\frac{a^*}{2},\frac{a^*}{2})\times(-\frac{a^*}{2},\frac{a^*}{2})\times(-\frac{a^*}{4},\frac{a^*}{4})$. Let $(u_1, u_2)$ be any non-negative minimizer of $\hat{e}(a_1, a_2, \beta)$ with $(a_1, a_2, \beta)\in H$. Using (\ref{eq1.10}) together with Cauchy's inequality, we have
\begin{align*}
	\hat{e}(a_1,a_2,\beta)=E_{a_1,a_2,\beta}(u_1,u_2)&\geq\sum_{i=1}^{2}(\frac{a^{*}-|a_{i}|}{2}-\frac{|\beta|}{2})\int_{\mathbb{R}}|u_{i}|^{4}dx\\&\geq\frac{a^*}{8}\sum_{i=1}^2\int_{\mathbb{R}}|u_i|^4dx.
\end{align*}
This indicates that
\begin{equation}\label{uniformly}
\mathrm{the~}L^4\text{-norm of minimizers of }\hat{e}(a_1, a_2, \beta)\text{ is uniformly bounded  in }H.	  
\end{equation}
We now let $(u_{i1},u_{i2})$ be a non-negative minimizer of $\hat{e}(a_{i1},a_{i2},\beta_{i})$ with $(a_{i1},a_{i2},\beta_i)\in H$, where $i=1, 2$. Then
\begin{align}\label{fangsuo}\hat{e}(a_{11},a_{12},\beta_{1})&=E_{a_{11},a_{12},\beta_{1}}(u_{11},u_{12}) \notag\\&=E_{a_{21},a_{22},\beta_2}(u_{11},u_{12})+\frac{a_{21}-a_{11}}{2}\int_{\mathbb{R}}|u_{11}|^4dx\notag\\&\quad+\frac{a_{22}-a_{12}}{2}\int_{\mathbb{R}}|u_{12}|^{4}dx+(\beta_{2}-\beta_{1})\int_{\mathbb{R}}|u_{11}|^{2}|u_{12}|^{2}dx\\&\geq\hat{e}(a_{21},a_{22},\beta_{2})+O(|(a_{21},a_{22},\beta_{2})-(a_{11},a_{12},\beta_{1})|).\notag\end{align}
In the same way, we also obtain
\begin{equation*}
	\hat{e}(a_{21},a_{22},\beta_2)\geq\hat{e}(a_{11},a_{12},\beta_1)+O(|(a_{21},a_{22},\beta_2)-(a_{11},a_{12},\beta_1)|),
\end{equation*}
which together with (\ref{fangsuo}) gives that
\begin{equation*}
	\lim_{(a_{11}, a_{12},\beta_{1})\to(a_{21},a_{22},\beta_{2})}\hat{e}(a_{11},a_{12},\beta_{1})=\hat{e}(a_{21},a_{22},\beta_{2}).
\end{equation*}
This implies that $\hat{e}(a_1,a_2,\beta)$  is continuous on $H$, and  the claim above is thus proved.

For any $(a_1, a_2, \beta)\in H$, denote $(u_{a_{1},\beta},u_{a_{2},\beta})$ as a non-negative minimizer of $\hat{e}(a_{1},a_{2},\beta)$. It is found that  $(u_{a_{1},\beta},u_{a_{2},\beta})$ satisfies the following Euler-Lagrange system
\begin{equation}\label{eq2.15}
	\begin{cases}(\sqrt{-\Delta}+V_1(x)-\mu_{a_1,\beta})u_{a_1,\beta}-a_1u_{a_1,\beta}^3-\beta u_{a_2,\beta}^2u_{a_1,\beta}=0\quad\mathrm{in}\quad \mathbb{R},\\(\sqrt{-\Delta}+V_2(x)-\mu_{a_2,\beta})u_{a_2,\beta}-a_2u_{a_2,\beta}^3-\beta u_{a_1,\beta}^2u_{a_2,\beta}=0 \quad \mathrm{in}\quad \mathbb{R},&\end{cases}
\end{equation}
that is, 
\begin{equation}\label{2.16}
	G_i(u_{a_1,\beta},u_{a_2,\beta},\mu_{a_i,\beta},a_i,\beta)=0,\quad i=1,2,
\end{equation}
where $(\mu_{a_1,\beta},\mu_{a_2,\beta})\in\mathbb{R}^2$ represents a Lagrange multiplier. Applying (\ref{uniformly}) and the above claim, we have
\begin{align*}\label{2.17}&E_{0,0,0}(u_{a_1,\beta},u_{a_2,\beta})\notag\\&=E_{a_1,a_2,\beta}(u_{a_1,\beta},u_{a_2,\beta})+\sum_{i=1}^2\frac{a_i}{2}\int_{\mathbb{R}}|u_{a_i,\beta}|^4dx+\beta\int_{\mathbb{R}}|u_{a_1,\beta}|^2|u_{a_2,\beta}|^2dx\\&=\hat{e}(a_1,a_2,\beta)+O(|(a_1,a_2,\beta)|)\to\hat{e}(0,0,0)\quad\mathrm{as}\quad(a_1,a_2,\beta)\to0.\notag\end{align*}

In addition, it is easy to verify that $\hat{e}(0,0,0)=\lambda_{11}+\lambda_{21}$, with $(\Psi_{11},\Psi_{21})$ being the unique non-negative minimizer of $\hat{e}(0,0,0)$, where $(\lambda_{i1},\Psi_{i1})$  \textcolor{red}{is defined in (\ref{lambda}) and the subsequent text, denoting} the first eigenpair of $\sqrt{-\Delta}+V_i(x)$ in $\mathcal{H}_i$  for $i=1, 2$. It then follows  from Lemma \ref{lemma2.1} that, for $ i=1,2$,
\begin{equation}\label{2.18}
	u_{a_{i},\beta}\to\Psi_{i1}\quad\mathrm{in}\quad\mathcal{H}_{i}\quad\mathrm{as}\quad(a_{1},a_{2},\beta)\to(0,0,0),
\end{equation} 
which together with (\ref{eq2.15}) implies that
\begin{align}\label{2.19} \mu_{a_{i},\beta}&\begin{aligned}=\int_{\mathbb{R}}|(-\Delta  ) ^{\frac{1}{4} } u_{a_i,\beta}|^2+V_i(x)u_{a_i,\beta}^2dx-a_i\int_{\mathbb{R}}|u_{a_i,\beta}|^4dx-\beta\int_{\mathbb{R}}|u_{a_1,\beta}|^2|u_{a_2,\beta}|^2dx\notag \end{aligned}\\&\to\lambda_{i1}\quad\mathrm{as}\quad(a_1,a_2,\beta)\to(0,0,0),\quad i=1,2.\end{align}
Combining  (\ref{2.18}) and (\ref{2.19}), we conclude that there exists a small constant $ \delta_1 > 0$ such that
\begin{equation*}
	\|u_{a_{i},\beta}-\Psi_{i1}\|_{\mathcal{H}_{i}}<\delta_{1}\quad\mathrm{and}\quad|\mu_{a_{i},\beta}-\lambda_{i1}|<\delta_{1}\quad\mathrm{if}\quad(a_{1},a_{2},\beta)\in B_{\delta_{1}}(\vec{0}),\quad i=1,2.
\end{equation*}
It thus follows from Lemma \ref{cunzai} and (\ref{2.16}) that
\begin{equation*}
	\mu_{a_i,\beta}=\mu_i(a_1,a_2,\beta), \ u_{a_i,\beta}=u_i(a_1,a_2,\beta),\quad\mathrm{if}\quad(a_1,a_2,\beta)\in B_{\delta_1}(\vec{0}),\quad i=1,2,
\end{equation*}
which implies that for suitably small $|(a_1,a_2,\beta)|$, $(u_{1}(a_{1},a_{2},\beta), u_{2}(a_{1},a_{2},\beta))$ is a unique non-negative minimizer of $\hat{e}(a_{1},a_{2},\beta)$. Therefore, the proof is completed.
\end{proof}

	\section{Optimal estimates} \label{section2}
In order to simplify the notations and the proof, for any fixed  $0<\beta<a^*$, we define $d_i=a_i+\beta>0$ $(i=1,2)$, and (\ref{eq1.9})   can then be rewritten as 
\begin{align}\label{new energy}E_{d_{1},d_{2}}(u_{1},u_{2})&:=\sum_{i=1}^{2}\int_{\mathbb{R}}\left(|(-\Delta  ) ^{\frac{1}{4} }u_i|^2+V_{i}(x)u_{i}^{2}-\frac{d_{i}}{2}|u_{i}|^{4}\right)dx\notag\\&\quad+\frac{\beta}{2}\int_{\mathbb{R}}\left(|u_{1}|^{2}-|u_{2}|^{2}\right)^{2}dx, \quad \text{for any } (u_1, u_2)\in \mathcal{X}. \end{align}
	 Correspondingly, the  problem (\ref{problem 1.7}) is equivalent to the following  problem
	\begin{equation} \label{new problem}
		e(d_1,d_2):=\inf_{\{(u_1,u_2)\in\mathcal{M}\}}E_{d_1,d_2}(u_1,u_2),
	\end{equation}
where $\mathcal{M}$ is given by (\ref{M}).

Moreover, we introduce the following single component minimization problem
\begin{equation}\label{single}
	e_i(d_i) := \inf_{\{ u \in \mathcal{H}_i, \ \|u\|_2^2 = 1 \}} E_{d_i}^i(u), \quad d_i > 0,
\end{equation}
where the energy functional  $E^i_{d_{i}}(u)$ is given by
\begin{equation}\label{single energy di}
	E_{d_i}^i(u) := \int_{\mathbb{R}} \left(|(-\Delta  ) ^{\frac{1}{4} }u|^2 + V_i(x)u^2 \right) dx - \frac{d_i}{2} \int_{\mathbb{R}} |u|^4 dx, \quad i=1 \ \text{or} \ 2.	
\end{equation}
Following  the proof of Theorem 1.3 in \cite{T25}, for  $i=1,2$ and $V_i(x)$ satisfying $(\mathcal{D})$, it is straightforward to deduce that $e_i(d_i)$ has at least one minimizer for any $0<d_i<a^*$, while $e_i(d_i)$ has no minimizers if $d_i\geq a^*$. The proof is somewhat standard, and we omit it here.\\
\noindent Also, for $i=1,2$, we set
\begin{equation}\label{max}
	p_i:=\max\{p_{ij}, j=1,\cdots, n_i \}>0,
\end{equation}
and let $\hat{\lambda}_{ij}\in(0,\infty]$ be given by 
\begin{equation*}
	\hat{\lambda}_{ij}=\left({p_i}\int_{\mathbb{R}}|x|^{p_i}Q^2(x)\mathrm{d}x\lim_{x\to x_{ij}}\frac{V_i(x)}{|x-x_{ij}|^{p_i}}\right)^{\frac{1}{1+p_i}},
\end{equation*}
where $j=1,\cdots, n_i$ and  $V_i(x)$ satisfies  (\ref{form 1}).\\
Moreover, for $i=1,2$, $j=1,\textcolor{red}{\cdots},n_i$, we denote $\hat{\lambda}_i=\min\{\hat{\lambda}_{i1},\cdots,\hat{\lambda}_{in_i}\}$ and define 
\begin{equation}\label{faltt}
	\hat{\mathcal{Z}}_i:=\{x_{ij}:\hat{\lambda}_{ij}=\hat{\lambda}_{i}\}
\end{equation}
as the locations of the flattest global minima of $V_i(x)$. 

 We now present the following  optimal
 energy estimates as $(d_1, d_2)\nearrow(a^*,a^*)$, which are crucial for proving  Theorem \ref{blow}.

 \begin{proposition}\label{pro 3.1}
 	Assume  $0<\beta<a^{*}$, $p_{ij}\in (0, 1)$, and $V_i(x)$  satisfies (\ref{form 1}) - (\ref{point}), where $i=1, 2$, $1\leq j\leq n_{i}$. Then there exist  $C_1$, $C_2>0$, independent of $d_1$ and $d_2$, such that
 	\begin{equation}\label{3.3}
 		C_1\left(a^*-\frac{d_1+d_2}{2}\right)^{\frac{p_0}{p_0+1}}\leq e(d_1,d_2)\leq C_2\left(a^*-\frac{d_1+d_2}{2}\right)^{\frac{p_0}{p_0+1}}\quad as\quad(d_1, d_2)\nearrow(a^*,a^*),
 	\end{equation}
 	where $p_0 > 0$ is defined by (\ref{index}). Furthermore, if $(u_{d_{1}},u_{d_{2}})$ is a  non-negative minimizer of $e(d_{1}, d_{2})$, then there exist $C_3$, $C_4>0$, independent of $d_1$ and $d_2$, such that for $i=1, 2$,
 	\begin{equation}\label{3.4}
 		C_3\left(a^*-\frac{d_1+d_2}{2}\right)^{-\frac{1}{p_0+1}}\leq\int_{\mathbb{R}}|u_{d_i}|^4dx\leq C_4\left(a^*-\frac{d_1+d_2}{2}\right)^{-\frac{1}{p_0+1}}
 	\end{equation}
 	as $(d_1, d_2)\nearrow(a^*,a^*)$.
 \end{proposition}
 
In the sequel, we focus on proving Proposition \ref{pro 3.1} and   begin with the   following lemma.
\begin{lemma}\label{single energy}
	Suppose that $V_1(x)$ and $V_2(x)$ satisfy $(\mathcal{D})$ and  (\ref{form 1}). If    \textcolor{red}{$0<p_{ij}<1$ for $i=1, 2$ and $j=1,\cdots, n_i$}, then there exist two positive constants $M_{1}<M_{2}$, independent of $d_1$ and $d_2$, such that
	\begin{equation}\label{bounded}
		M_1(a^*-d_i)^{\frac{p_i}{p_i+1}}\leq e_i(d_i)\leq M_2(a^*-d_i)^{\frac{p_i}{p_i+1}} \quad \text{for} \quad
		  0<d_i<a^*,
	\end{equation}
	where $p_i$ is defined in (\ref{max}).
\end{lemma}
\begin{proof}
	From Lemma 2.5 of \cite{T20}, we know that for $i=1,2$, $e_i(d_i)$ is uniformly continuous with respect to $d_i$ on $(0,a^{*})$, and thus $e_i(d_i)$ is uniformly bounded for $0< d_i< a^*$. Moreover, since $e_i(d_i)$ is decreasing with respect to  $d_i$ on $(0,a^{*})$, it thus suffices to consider the case where $d_i$ is close to $a^*$ for $i=1, 2$. We first address the lower bound.
	
	For any $\hat{\gamma}_i>0$  $(i=1,2)$ and $u\in H^{\frac{1}{2}}(\mathbb{{R}})$ with $ \|u\|_2^2=1$, using  (\ref{eq1.10}), a direct calculation gives that
	\begin{align}\label{eq 3.9}E_{d_i}^i(u)&\geqslant\int_{\mathbb{R}}V_i(x)u^{2}dx+\frac{a^{*}-d_i}{2}\int_{\mathbb{R}}|u(x)|^{4}dx \notag  \\&=\hat{\gamma}_i+\int_{\mathbb{R}}\left[(V_i(x)-\hat{\gamma}_i)u^{2}+\frac{a^{*}-d_i}{2}|u(x)|^4\right]dx\\&\geqslant\hat{\gamma}_i-\frac{1}{2(a^*-d_i)}\int_{\mathbb{R}}\left[\hat{\gamma}_i-V_i(x)\right]_{+}^{2}dx,\quad i=1, 2,\notag \end{align}
	where $[\ \cdot \ ]_+=\max\{0,\cdot\}$ denotes the positive part. For $i=1,2$, $j=1,\textcolor{red}{\cdots},n_i$ and small enough $\hat{\gamma}_i$, the set 
	\begin{equation*}
		\{x\in\mathbb{R}:V_i(x)<\hat{\gamma}_i\}
	\end{equation*}
	is contained in the disjoint union of $n_i$ balls of radius at most $k_i\hat{\gamma}_i^{\frac{1}{p_i}}$, centered at $x_{ij}$, for a suitable $k_i>0$. Furthermore, $V_i(x)\geq(\frac{|x-x_{ij}|}{k_i})^{p_{i}}$ on these balls for $i=1, 2$,  $j=1,\textcolor{red}{\cdots},n_i$. Therefore,
	\begin{equation}\label{eq 3.10}
		\int_{\mathbb{R}}\left[\hat{\gamma}_i-V_i(x)\right]_{+}^{2}dx\leq n_i\int_{\mathbb{R}}\left[\hat{\gamma}_i-(\frac{\left|x\right|}{k_i})^{p_i}\right]_{+}^{2}dx\leq 2n_ik_i\hat{\gamma}_i^{2+\frac{1}{p_i}},\quad i=1,2.
	\end{equation}
	
	Let $\hat{\gamma}_i=C_i(a^*-d_i)^{\frac{p_i}{p_i+1}}$, where $C_i$ is a small positive constant for $i=1,2$. It then follows from (\ref{eq 3.9}) and (\ref{eq 3.10}) that there exists  $M_1>0$ such that
	\begin{equation*}
		e_i(d_i)\geq M_1(a^*-d_i)^{\frac{p_i}{p_i+1}} ,\quad i=1, 2.
	\end{equation*}
	
Before turning to the proof of the upper bound in (\ref{bounded}), we define 
	\begin{equation}\label{phi}
		\phi(x)=A_{R\tau}\frac{\tau^\frac{1}{2}}{\|Q\|_2}\eta\left(\frac{x-x_0}{R}\right)Q\left(\tau(x-x_0)\right),
	\end{equation}
	where $x_0\in \mathbb{{R}}$, $\tau>0$, $R>0$, $Q$ is the unique  radially symmetric ground state of (\ref{classical eq}), $\eta(x)\in  C_0^\infty(\mathbb{R})$ is a non-negative cut-off function satisfying
	\begin{equation*}\label{3.12}
		\eta(x)=1\quad\mathrm{for}\quad|x|\leq1;\quad \eta(x)=0\quad\mathrm{for}\quad|x|>2;\quad0\leq\eta(x)\leq1.
	\end{equation*}
	Here, $A_{R\tau}$ is chosen such that $\int_\mathbb{R}\phi^2 dx=1$. It then follows from the polynomial decay of $Q$ (\ref{decay}) that
	\begin{equation}\label{ARt}
		\frac{1}{A_{R\tau}^{2}}=\frac{1}{\|Q\|_{2}^{2}}\int_{\mathbb{R}}\eta^{2}\left(\frac{x}{R\tau}\right)Q^{2}(x)dx=1+O((R\tau)^{-3})\quad \mathrm{as}\quad R\tau\to\infty.
	\end{equation}
	Combining this with Lemma 3.1 in \cite{T21}, (\ref{decay})  and (\ref{eq1.11}), if we fix $R >0$, we  obtain that, for $i=1, 2$,
	\begin{align}\label{3.14}&\int_{\mathbb{R}}|(-\Delta)^{\frac{1}{4}}\phi(x)|^{2}  dx-\frac{d_i}{2}\int_{\mathbb{R}}|\phi(x)|^{4} dx\notag\\&=\frac{\tau}{\|Q\|_2^2}\left[\int_{\mathbb{R}}|(-\Delta)^{\frac{1}{4}}Q|^2 dx -\frac{d_i}{2\|Q\|_{2}^{2}}\int_{\mathbb{R}}|Q|^{4} dx+O(\tau^{-2})\right]
	\notag	\\	&=\frac{\tau}{2\|Q\|_{2}^{2}}\left[\left(1-\frac{d_i}{\|Q\|_{2}^{2}}\right)\int_{\mathbb{R}}|Q|^{4}dx+O(\tau^{-2})\right]\\&= \left(1-\frac{d_i}{\|Q\|_2^2}\right)\tau+O(\tau^{-1})\quad \mathrm{as} \quad\tau\to\infty.\notag\end{align}

	Next, we  prove the upper bound in (\ref{bounded}). For $i=1, 2$, using the trial function $\phi(x)$ defined in (\ref{phi}) and picking $x_0\in \hat{\mathcal{Z}}_i$, where $\hat{\mathcal{Z}}_i$ is defined in (\ref{faltt}),  we can fix $R>0$ small enough such that
	\begin{equation*}
		V_i(x)\leqslant C|x-x_0|^{p_i}\quad\mathrm{for}\quad|x-x_0|\leqslant2R,
	\end{equation*}in which case, we deduce that
	\begin{equation}\label{3.15}
		\int_{\mathbb{R}}V_i(x)\phi^2\mathrm{d}x\leqslant C\tau^{-p_i}A_{R\tau}^2\int_{\mathbb{R}}|x|^{p_i}Q^2(x)dx.
	\end{equation}
	From (\ref{ARt}), (\ref{3.14}) and (\ref{3.15}), we infer that for large $\tau$,
	\begin{equation}\label{limitiation}
		e_i(d_i)\leq\frac{\tau(a^*-d_i)}{2\|Q\|_2^4}\int_\mathbb{R}|Q(x)|^4dx+C\tau^{-p_i}\int_\mathbb{R}|x|^{p_i}Q^2(x)dx+O(\tau^{-1}),\quad i=1,2.
	\end{equation}
	By taking $\tau=(a^{*}-d_i)^{\frac{-1}{p_i+1}}$, the desired upper bound of $e_i(d_i)$ is thus obtained when \textcolor{red}{$0<p_{ij}<1$ for $i=1, 2$ and $j=1,\cdots, n_i$}. Therefore, the proof of this lemma is completed.
\end{proof}

For any fixed $0 < \beta < a^*$, let $(u_{d_1},u_{d_2})$  denote a non-negative minimizer of problem (\ref{new problem}). We establish the following energy estimates for $e(d_1,d_2)$.

\begin{lemma}\label{lemma 5.1}
	Let $e_i(d_i)$ be defined by (\ref{single}) for $i=1, 2$. Then, under the assumptions of Proposition \ref{pro 3.1}, there exists a constant $C > 0$, independent of $d_1$ and $d_2$, such that
	\begin{equation}\label{3.5}
		e_1(d_1)+e_2(d_2)+\frac{\beta}{2}\int_{\mathbb{R}}\left(|u_{d_1}|^2-|u_{d_2}|^2\right)^2dx\leq e(d_1,d_2)\leq C\left(a^*-\frac{d_1+d_2}{2}\right)^{\frac{p_0}{p_0+1}}
	\end{equation}
as $(d_1,d_2)\nearrow(a^*,a^*)$.
\begin{proof}
	Since $(u_{d_1},u_{d_2})$ is a non-negative minimizer of (\ref{new problem}), it follows from (\ref{single energy di}) that, for any $(u,v)\in\mathcal{X}$,
	\begin{equation*}
		E_{d_1,d_2}(u,v)=E_{d_1}^1(u)+E_{d_2}^2(v)+\frac{\beta}{2}\int_{\mathbb{R}}\left(|u|^2-|v|^2\right)^2dx,
	\end{equation*}
which indicates that 
\begin{equation*}
	e(d_1,d_2)\geq e_1(d_1)+e_2(d_2)+\frac{\beta}{2}\int_{\mathbb{R}}\left(|u_{d_1}|^2-|u_{d_2}|^2\right)^2dx.
\end{equation*}
It thus proves the lower bound of (\ref{3.5}).
Next, we turn to the proof of the upper bound of (\ref{3.5}). \textcolor{red}{Without loss of generality}, given $\bar{p}_{1}$ and $p_0$ as  in (\ref{index}), we may suppose that $p_0=\bar{p}_{1}=\min\left\{p_{11}, p_{21}\right\}>0$ and  $p_{11}\leq p_{21}$. Consider the trial function ($\phi$, $\phi$) defined in (\ref{phi}) with $x_0=x_{11}$ and choose $R>0$ small enough such that, for $i=1,2$,
\begin{equation*}
	V_i(x)\leqslant C|x-x_{11}|^{p_{i1}}\quad\mathrm{for}\quad|x-x_{11}|\leqslant2R.
\end{equation*}
It then follows from (\ref{decay}) that,
\begin{equation*}
\int_{\mathbb{R}}V_i(x)\phi^2\mathrm{d}x\leqslant C\tau^{-p_{i1}}A_{R\tau}^2\int_{\mathbb{R}}|x|^{p_{i1}}Q^2(x)\mathrm{d}x \leqslant	C\tau^{-p_{i1}}\quad\mathrm{as}\quad\tau\to\infty,\quad i=1, 2,
\end{equation*}
which together with (\ref{3.14}) implies that 
\begin{equation}\label{3.200}
	E_{d_1,d_2}(\phi,\phi)\leq\frac{2}{a^*}\left(a^*-\frac{d_1+d_2}{2}\right)\tau+C\left(\tau^{-p_{11}}+\tau^{-p_{21}}\right).
\end{equation}
By the fact that $p_{0}=p_{11}\leq p_{21}$ and taking $\tau=\left(a^*-\frac{d_1+d_2}{2}\right)^{\frac{-1}{p_0+1}}$ in (\ref{3.200}), we derive that
\begin{equation*}
	e(d_1,d_2)\leq E_{d_1,d_2}(\phi,\phi)\leq C\left(a^*-\frac{d_1+d_2}{2}\right)^{\frac{p_0}{p_0+1}},
\end{equation*}
which thus yields the  upper bound of (\ref{3.5}). This completes the proof of Lemma \ref{lemma 5.1}.

  It is  known from Lemma 5.2 of  \cite{T18} that the energy estimates of $e(d_1)$ and $e(d_2)$ are crucial for deriving $L^4(\mathbb{{R}})$ - estimates of $(u_{d_1},u_{d_2})$. We  thus apply Lemma \ref{single energy} and  Lemma \ref{lemma 5.1} to establish the following estimates for the minimizers.

\begin{lemma}\label{L4}
Under the assumptions stated in Proposition \ref{pro 3.1}, we have
\begin{equation}\label{L4 1}
	C\left(a^*-\frac{d_1+d_2}{2}\right)^{-\frac{1}{p_0+1}\frac{p_0}{p_i}}\leq\int_{\mathbb{R}}|u_{d_i}|^4dx\leq\frac{1}{C}\left(a^*-\frac{d_1+d_2}{2}\right)^{-\frac{1}{p_0+1}}\quad as\quad(d_1,d_2)\nearrow(a^*,a^*),
\end{equation}
and
\begin{equation}\label{L4 2}
	\lim_{(d_1, d_2)\nearrow(a^*,a^*)}\frac{\int_{\mathbb{R}}|u_{d_1}|^4dx}{\int_{\mathbb{R}}|u_{d_2}|^4dx}=1,
\end{equation}
where $p_i\geq p_0$, and $p_i>0$ is defined in (\ref{max}), $i=1,2$.
\end{lemma}
\begin{proof}
	We first consider the lower bound of (\ref{L4 1}). To this end, select any $0<a<d_{i}<a^{*}(i=1,2)$, and see that, for $i=1,2$,
	\begin{equation}\label{3.24}
		e_{i}(a)\leq E_{d_{i}}^{i}(u_{d_{i}})+\frac{d_{i}-a}{2}\int_{\mathbb{R}}|u_{d_{i}}|^{4}dx.
	\end{equation}
\textcolor{red}{By (\ref{new problem}) and (\ref{single energy di}), we have 
	\begin{equation*}
		e(d_1,d_2)=E_{d_1}^1(u_{d_1})+E_{d_2}^2(u_{d_2})+\frac{\beta}{2}\int_{\mathbb{R}}\left(|u_{d_1}|^{2}-|u_{d_2}|^{2}\right)^{2}dx.\end{equation*}}
\textcolor{red}{It then follows} from Lemma \ref{lemma 5.1} that, for $i=1,2$,
\begin{equation}\label{3.25}
	e_i(d_i)\leq E_{d_i}^i(u_{d_i})\leq C\left(a^*-\frac{d_1+d_2}{2}\right)^{\frac{p_0}{p_0+1}},
\end{equation}
which combined with (\ref{bounded}) and (\ref{3.24}) gives that
\begin{align}\label{3.26}\frac{1}{2}\int_{\mathbb{R}}|u_{d_i}|^4dx&\geq\frac{e_i(a)-C\left(a^*-\frac{d_1+d_2}{2}\right)^{\frac{p_0}{p_0+1}}}{d_i-a}\notag\\&\geq\frac{M_1(a^*-a)^{\frac{p_i}{p_i+1}}-C\left(a^*-\frac{d_1+d_2}{2}\right)^{\frac{p_0}{p_0+1}}}{d_i-a}.\end{align}
Take $a=d_{i}-C_{1}\left(a^{*}-\frac{d_{1}+d_{2}}{2}\right)^{\frac{p_{0}(p_{i}+1)}{p_{i}(p_{0}+1)}}$, where $C_1>0$ is large enough such that
$M_1C_{1}^{\frac{p_{i}}{p_{i}+1}}>2C$. It then follows from (\ref{3.26}) that
\begin{equation}\label{3.27}
	\int_{\mathbb{R}}|u_{d_i}|^4dx\geq C_2\left(a^*-\frac{d_1+d_2}{2}\right)^{-\frac{1}{p_0+1}\frac{p_0}{p_i}},\quad i=1,2.
\end{equation}
Hence,  the lower bound of (\ref{L4 1}) is proven.

We next focus on  the upper bound of (\ref{L4 1}).  \textcolor{red}{Without loss of generality}, we can suppose that $d_{1}\leq d_{2}\leq a^{*}$ and $(d_{1},d_{2})\neq(a^{*},a^{*})$. Using (\ref{eq1.10}), we then get that
\begin{equation*}
	E_{d_1}^1(u_{d_1})\geq\frac{a^*-d_1}{2}\int_{\mathbb{R}}|u_{d_1}|^4dx\geq\frac{1}{2}\left(a^*-\frac{d_1+d_2}{2}\right)\int_{\mathbb{R}}|u_{d_1}|^4dx,
\end{equation*}
which together with (\ref{3.25}) implies that the upper bound of (\ref{L4 1}) holds for $u_{d_1}$. Similarly, the upper bound of (\ref{L4 1}) also holds  for $u_{d_2}$
when (\ref{L4 2})  holds, and then the proof is completed.

We now turn to prove (\ref{L4 2}). According to Lemma \ref{lemma 5.1}, we have
\begin{equation*}
	\int_{\mathbb{R}}\left(|u_{d_1}|^2-|u_{d_2}|^2\right)^2dx\leq C\left(a^*-\frac{d_1+d_2}{2}\right)^{\frac{p_0}{p_0+1}}\quad\mathrm{as}\quad(d_1,d_2)\nearrow(a^*,a^*),
\end{equation*}
which indicates that
\begin{align*}&\bigg|\int_{\mathbb{R}}|u_{d_{1}}|^{4}dx-\int_{\mathbb{R}}|u_{d_{2}}|^{4}dx\bigg|\notag\\&\leq\left(\int_{\mathbb{R}}\left(|u_{d_1}|^2-|u_{d_2}|^2\right)^2dx\right)^{\frac{1}{2}}\left(\int_{\mathbb{R}}\left(|u_{d_1}|^2+|u_{d_2}|^2\right)^2dx\right)^{\frac{1}{2}}\\&\leq C\left(a^{*}-\frac{d_{1}+d_{2}}{2}\right)^{\frac{p_{0}}{2(p_{0}+1)}}\left[\left(\int_{\mathbb{R}}|u_{d_{1}}|^{4}dx\right)^{\frac{1}{2}}+\left(\int_{\mathbb{R}}|u_{d_{2}}|^{4}dx\right)^{\frac{1}{2}}\right]\notag\end{align*}
as $(d_1,d_2)\nearrow(a^*,a^*)$. Noting from (\ref{3.27}) that $\int_{\mathbb{R}}|u_{d_i}|^4dx\to\infty$ as $(d_1,d_2)\nearrow(a^*,a^*)$ for $i=1,2$, we then deduce from the above estimate that (\ref{L4 2}) holds.

Next, we show that the upper estimates of (\ref{3.5})  and (\ref{L4 1}) are optimal.
From Lemma \ref{lemma 5.1}, we have
\begin{equation}\label{3.30}
	\sum_{i=1}^2\int_{\mathbb{R}}V_i(x)u_{d_i}^2dx\leq e(d_1,d_2)\leq C\left(a^*-\frac{d_1+d_2}{2}\right)^{\frac{p_0}{p_0+1}}\quad\mathrm{as}\quad(d_1,d_2)\nearrow(a^*,a^*).
\end{equation}
Define
\begin{equation}\label{3.31}
	\epsilon^{-1}(d_1,d_2):=\int_{\mathbb{R}}|u_{d_1}(x)|^4dx,\quad\mathrm{where}\quad\epsilon(d_1,d_2)>0.
\end{equation}
Then, it follows from (\ref{L4 1}) that $\epsilon(d_{1},d_{2})\to0$ as $(d_{1},d_{2})\nearrow(a^{*},a^{*})$. Furthermore, using (\ref{3.5}) and (\ref{L4 2}), we infer that, for $i=1,2$,
\begin{equation}\label{3.32}
	\int_{\mathbb{R}}|u_{d_i}|^4dx,\int_{\mathbb{R}} | (-\Delta  ) ^{\frac{1}{4} }u_{d_i}|^{2}dx\sim\epsilon^{-1}(d_1,d_2)\quad\mathrm{as}\quad(d_1,d_2)\nearrow(a^*,a^*),
\end{equation}
here we denote $f\sim g$ to mean that $f/g$ is bounded from below and above. We now set the $L^{2}(\mathbb{R})$-normalized function
\begin{equation*}\label{3.33}
\tilde{w}_{d_i}(x):=\epsilon^{\frac{1}{2}}(d_1,d_2)u_{d_i}\left(\epsilon(d_1,d_2)x\right),\quad i=1,2.
\end{equation*}
Then, by (\ref{L4 2}), (\ref{3.31}) and (\ref{3.32}), we derive that, for $i=1,2$,
\begin{equation}\label{3.34}
	\int_{\mathbb{R}}|\tilde{w}_{d_i}|^4dx=1\quad\mathrm{and}\quad M\leq\int_{\mathbb{R}}|(-\Delta)^{\frac{1}{4}}\tilde{w}_{d_i}|^2dx\leq\frac{1}{M}\quad\mathrm{as}\quad(d_1,d_2)\nearrow(a^*,a^*),
\end{equation}
where $M>0$ is independent of $d_1$ and $d_2$. For simplicity, in the remainder of this section we adopt $\epsilon>0$ to denote  $\epsilon(d_{1},d_{2})$ so that $\epsilon\to0$ as $(d_1,d_2)\nearrow(a^*,a^*)$.

\begin{lemma}\label{lemma 3.5}
	Under the assumptions stated in Proposition \ref{pro 3.1}, we have \\
	(i). There exists a sequence $\left \{ {y_{\epsilon}} \right \}$ and  positive constants $R_0$, $\eta_i$, such that the function
	\begin{equation}\label{3.35}
		w_{d_i}(x)=\tilde{w}_{d_i}(x+y_\epsilon)=\epsilon^{ \frac{1}{2}}u_{d_i}(\epsilon x+\epsilon y_\epsilon)
	\end{equation}
	satisfies
	\begin{equation}\label{3.36}
		\operatorname*{liminf}_{\epsilon\to0}\int_{B_{R_0}(0)}|w_{d_i}|^2dx\geq\eta_i>0,\quad where \quad i=1,2.
	\end{equation}\\
	(ii). For $\Lambda$ defined by (\ref{Lamda}), the following estimate is satisfied
	\begin{equation}\label{3.37}
		\lim_{(d_1,d_2)\nearrow(a^*,a^*)}\mathrm{dist}(\epsilon y_\epsilon,\Lambda)=0.
	\end{equation}
Furthermore, for any sequence $\{(d_{1k},d_{2k})\}$ with $(d_{1k},d_{2k})\nearrow(a^{*},a^{*})$ as $k\to\infty$,  there exists a
convergent subsequence
 of $\{(d_{1k},d_{2k})\}$, still denoted by $\{(d_{1k},d_{2k})\}$, such that, for $i=1,2$,
\begin{equation}\label{3.38}
	\epsilon_ky_{\epsilon_k}\xrightarrow{k} x_0\in\Lambda\quad \text{and}\quad w_{d_{ik}}\xrightarrow{k} w_0\quad\textit{strongly in H}^{\frac{1}{2}}(\mathbb{R}),
\end{equation}
where $w_0$ satisfies
\begin{equation}\label{3.39}
	w_0(x)=\frac{\lambda^{\frac{1}{2}}}{\|Q\|_2}Q(\lambda(x-y_0)) \quad \text{for some}\ y_0\in \mathbb{R} \ and\ \lambda>0.
\end{equation}
\begin{proof} \textbf{(i).}
	To obtain (\ref{3.36}), we know from (\ref{3.35}) that it suffices to prove that there exist positive constants $R_0$ and  $\eta_i$ such that
	\begin{equation}\label{3.40}
		\operatorname*{liminf}_{\epsilon\to0}\int_{B_{R_0}(y_\epsilon)}|\tilde{w}_{d_i}|^2dx\geq\eta_i>0,\quad i=1,2.
	\end{equation}
We initially prove that (\ref{3.40}) holds for $\tilde{w}_{d_1}$. To argue by contradiction, we assume this does not hold. Then, for any $R>0$, there exists a subsequence $\{\tilde{w}_{d_{1k}}\}$, where $(d_{1k},d_{2k})\nearrow(a^{*},a^{*})$ as $k\to\infty$, such that
\begin{equation*}
	\lim_{k\to\infty}\sup_{y\in\mathbb{R}}\int_{B_R(y)}|\tilde{w}_{d_{1k}}|^2dx=0.
\end{equation*}
It then follows from  the vanishing lemma  \cite[Lemma I.1]{vanish} that $\tilde{w}_{d_{1k}}\overset{k}{\operatorname*{\to}}0$  strongly in $L^{p}(\mathbb{R})$ with $p\in(2, \infty)$, thereby yielding a contradiction with (\ref{3.34}). Hence, (\ref{3.40}) holds for $\tilde{w}_{d_1}$ with respect to a sequence ${\{y_{\epsilon}\}}
$, and positive constants $R_0$,  $\eta_1$.

For the sequence ${\{y_{\epsilon}\}}
$ and positive constants $R_0$,  $\eta_1$ obtained above,  we now prove that
(\ref{3.40}) holds for $\tilde{w}_{d_2}$ with some positive constant $\eta_2$.  To argue by contradiction, we assume  this fails. Then,  there exists a subsequence $\{\tilde{w}_{d_{2k}}\}$, where $(d_{1k},d_{2k})\nearrow(a^{*},a^{*})$ as $k\to\infty$, such that
\begin{equation*}
	\operatorname*{lim\sup}_{k\to\infty}\int_{B_{R_0}(y_{\epsilon_k})}|\tilde{w}_{d_{2k}}|^2dx=0,
\end{equation*}
where $\epsilon_{k}:=\epsilon(d_{1k},d_{2k})>0$ is given by (\ref{3.31}). Since $\tilde{w}_{d_{i}}$ is uniformly bounded  in $H^{\frac{1}{2}}(\mathbb{R})\cap L^{r}(\mathbb{R})$ with $r\in [2, \infty)$, one can select $r>4$ and $\theta\in(0,1)$ so that $\frac{1}{4}=\frac{1-\theta}{r}+\frac{\theta}{2}$. Then, by the H\"{o}lder inequality, we derive that 
\begin{align*}\label{3.41}\int_{B_{R_{0}}(y_{\epsilon_{k}})}|\tilde{w}_{d_{1k}}|^{2}|\tilde{w}_{d_{2k}}|^{2}dx&\leq\left(\int_{B_{R_{0}}(y_{\epsilon_{k}})}|\tilde{w}_{d_{1k}}|^{4}dx\right)^{\frac{1}{2}}\left(\int_{B_{R_{0}}(y_{\epsilon_{k}})}|\tilde{w}_{d_{2k}}|^{4}dx\right)^{\frac{1}{2}}\notag\\&\leq C\left(\int_{B_{R_0}(y_{\epsilon_k})}|\tilde{w}_{d_{2k}}|^r dx\right)^{\frac{2(1-\theta)}{r}}\left(\int_{B_{R_0}(y_{\epsilon_k})}|\tilde{w}_{d_{2k}}|^2dx\right)^\theta\\&\leq C\left(\int_{B_{R_{0}}(y_{\epsilon_{k}})}|\tilde{w}_{d_{2k}}|^{2}dx\right)^{\theta}\to0\quad\mathrm{as}\quad k\to\infty.\notag \end{align*}
By using the estimate (\ref{3.40}) for 
$\tilde{w}_{d_{1k}}$, we employ the H\"{o}lder inequality again to deduce from the above that, for $k$ large, 
\begin{align*}\int_{B_{R_{0}}(y_{\epsilon_{k}})}\left(|\tilde{w}_{d_{1k}}|^{2}-|\tilde{w}_{d_{2k}}|^{2}\right)^{2}dx&\geq\frac{1}{2}\int_{B_{R_{0}}(y_{\epsilon_{k}})}|\tilde{w}_{d_{1k}}|^{4}dx\\&\geq\frac{1}{4 R_{0}}\left(\int_{B_{R_{0}}(y_{\epsilon_{k}})}|\tilde{w}_{d_{1k}}|^{2}dx\right)^{2}\geq\frac{\eta_{1}^{2}}{4R_{0}}.\end{align*}
This leads to a contradiction,  since Lemma \ref{lemma 5.1} yields that
\begin{equation*}
	\int_{\mathbb{R}}\left(|\tilde{w}_{d_{1k}}|^2-|\tilde{w}_{d_{2k}}|^2\right)^2dx=\epsilon_k\int_{\mathbb{R}}\left(|u_{d_{1k}}|^2-|u_{d_{2k}}|^2\right)^2dx\to0\quad\mathrm{as}\quad k\to\infty.
\end{equation*}
Consequently,  (\ref{3.40}) also holds for $\tilde{w}_{d_{2}}$ with some positive constant $\eta_2$, which thus completes the proof of part (i).

\textbf{(ii).} To begin with, we prove that (\ref{3.37}) holds true.  From (\ref{3.30}), we deduce that 
\begin{equation}\label{3.41*}
	\sum_{i=1}^2\int_{\mathbb{R}}V_i(x)u_{d_i}^2dx=\sum_{i=1}^2\int_{\mathbb{R}}V_i(\epsilon x+\epsilon y_\epsilon)w_{d_i}^2dx\to0 \quad as \quad (d_1,d_2)\nearrow(a^*,a^*). 
\end{equation}
Indeed, if (\ref{3.37}) is false, then there exists $\delta>0$ and a subsequence $\{(d_{1n},d_{2n})\}$ satisfying $(d_{1n},d_{2n})\nearrow(a^*,a^*)$ as $n\rightarrow\infty$, such that
\begin{equation*}
	\epsilon_n:=\epsilon(d_{1n},d_{2n})\to0\quad\mathrm{and}\quad\mathrm{dist}(\epsilon_ny_{\epsilon_n},\Lambda)\geq\delta>0\quad\mathrm{as}\quad n\to\infty.
\end{equation*}
This indicates   that there exists $C = C(\delta) > 0$ satisfying
\begin{equation*}
	\lim_{n\to\infty}\left(V_1(\epsilon_ny_{\epsilon_n})+V_2(\epsilon_ny_{\epsilon_n})\right)\geq C(\delta)>0.
\end{equation*}
It then follows from Fatou's Lemma and (\ref{3.36}) that
\begin{align*}&\operatorname*{lim}_{n\to\infty}\sum_{i=1}^{2}\int_{\mathbb{R}}V_{i}(\epsilon_{n}x+\epsilon_{n}y_{\epsilon_{n}})w_{d_{in}}^{2}dx\\&\geq\sum_{i=1}^{2}\int_{B_{R_{0}}(0)}\operatorname*{liminf}_{n\to\infty}V_{i}(\epsilon_{n}x+\epsilon_{n}y_{\epsilon_{n}})w_{d_{in}}^{2}dx\geq C(\delta)\operatorname*{min}\{\eta_{1},\eta_{2}\}.\end{align*}
This leads to a contradiction with (\ref{3.41*}),  and thus (\ref{3.37}) is proved.

Next, we proceed to prove (\ref{3.38}) and (\ref{3.39}). As a non-negative minimizer of  (\ref{new problem}), $(u_{d_1},u_{d_2})$ satisfies the following Euler-Lagrange system
\begin{equation}\label{3.43}
	\begin{cases}\sqrt{-\Delta} u_{d_1}+V_1(x)u_{d_1}=\mu_{d_1}u_{d_1}+d_1u_{d_1}^3-\beta(u_{d_1}^2-u_{d_2}^2)u_{d_1}&\mathrm{in}\quad\mathbb{R},\\\sqrt{-\Delta}u_{d_2}+V_2(x)u_{d_2}=\mu_{d_2}u_{d_2}+d_2u_{d_2}^3-\beta(u_{d_2}^2-u_{d_1}^2)u_{d_2}&\mathrm{in}\quad\mathbb{R}.\end{cases}
\end{equation}
Here $(\mu_{d_{1}},\mu_{d_{2}})$ denotes a suitable Lagrange multiplier satisfying
\begin{equation*}
	\mu_{d_i}=E_{d_i}^i(u_{d_i})-\frac{d_i}{2}\int_{\mathbb{R}}|u_{d_i}|^4dx-\beta\int_{\mathbb{R}}(-1)^i(u_{d_1}^2-u_{d_2}^2)u_{d_i}^2dx,\quad i=1,2.
\end{equation*}
We then deduce from  Lemma \ref{lemma 5.1}, (\ref{new energy}), (\ref{L4 2}) and (\ref{3.32}) that, for $i=1,2$,
\begin{equation}\label{3.44}
	\mu_{d_{i}}\sim-\frac{d_{i}}{2}\int_{\mathbb{R}}|u_{d_{i}}|^{4}dx\sim-\epsilon^{-1}\quad\mathrm{and}\quad\mu_{d_{1}}/\mu_{d_{2}}\to1\quad\mathrm{as}\quad(d_{1},d_{2})\nearrow(a^{*},a^{*}).
\end{equation}
From the definition of $w_{d_i}$ given in (\ref{3.35}), the functions satisfy the following system
\begin{equation}\label{3.45}
\begin{cases}\sqrt{-\Delta} w_{d_1}+\epsilon V_1(\epsilon x+\epsilon y_\epsilon)w_{d_1}=\epsilon\mu_{d_1}w_{d_1}+d_1w_{d_1}^3-\beta(w_{d_1}^2-w_{d_2}^2)w_{d_1}&\mathrm{in} \quad\mathbb{R},\\ \sqrt{-\Delta} w_{d_2}+\epsilon V_2(\epsilon x+\epsilon y_\epsilon)w_{d_2}=\epsilon \mu_{d_2}w_{d_2}+d_2w_{d_2}^3-\beta(w_{d_2}^2-w_{d_1}^2)w_{d_2}&\mathrm{in}\quad\mathbb{R},\end{cases}
\end{equation}
where the Lagrange multiplier $(\mu_{d_{1}},\mu_{d_{2}})$ satisfies (\ref{3.44}).

 For any sequence $\{(d_{1k}, d_{2k})\}$ where $(d_{1k}, d_{2k}) \nearrow (a^*, a^*)$ as $k \rightarrow \infty$, it can be inferred from (\ref{3.37}) and (\ref{3.44}) that there exists a subsequence of $\{(d_{1k}, d_{2k})\}$, which we still denote by $\{(d_{1k}, d_{2k})\}$, such that, for $i=1,2$,
 \begin{equation}\label{3.46}
 	\epsilon_{k}y_{\epsilon_{k}}\overset{k}{\operatorname*{\to}}x_{0}\in\Lambda,\quad\epsilon_{k}\mu_{d_{ik}}\overset{k}{\operatorname*{\to}}-\lambda<0\quad\mathrm{for~some}\quad\lambda>0,
 \end{equation}
 and
 \begin{equation*}
 	w_{d_{ik}}\overset{k}{\operatorname*{\rightharpoonup}}w_i\geq0\text{ weakly in }H^{\frac{1}{2}}(\mathbb{R})\text{ for some }w_i\in H^{\frac{1}{2}}(\mathbb{R}).
 \end{equation*}
By (\ref{3.5}), we have
\begin{equation}\label{3.47}
	\|w_{d_1}^2-w_{d_2}^2\|_2=\epsilon^{\frac12}\|u_{d_1}^2-u_{d_2}^2\|_2\to0\quad\mathrm{as}\quad(d_1,d_2)\nearrow(a^*,a^*),
\end{equation}
which implies that $w_1=w_2\geq0$ a.e. in $\mathbb{{R}}$. We hence denote $0\leq w_{0}:=w_{1}=w_{2}\in H^{\frac{1}{2}}(\mathbb{R})$.
By passing to the weak limit in (\ref{3.45}), it follows from (\ref{3.46}) and (\ref{3.47}) that  $w_0$ satisfies 
\begin{equation}\label{3.48}
	\sqrt{-\Delta} w_0(x)=-\lambda w_0(x)+a^*w_0^3(x)\quad\mathrm{in}\quad\mathbb{R}.
\end{equation}
Moreover, similarly to the proof of Proposition 3.1 in \cite{T26}, applying the strict positivity of the kernel $G_{1/2,\lambda}(x-y)$ of the resolvent $(\sqrt{-\Delta}+\lambda)^{-1}$ on $\mathbb{{R}}$, we deduce that $w_{0}\equiv0$ or $w_{0}>0$ for $x\in \mathbb{{R}}$. It is then  concluded from (\ref{3.36}) that $w_{0}>0$.

Next, we come to prove that $w_0(x)$ satisfies (\ref{3.39}). From (\ref{3.48}), we know that $w_0$ is a  positive solution of 
\begin{equation}\label{3.50}
		\sqrt{-\Delta} u=-\lambda u+a^*u^3\quad\mathrm{in}\quad\mathbb{R}.
\end{equation}
Then we conclude that
\begin{equation*}
	W(x):=\left(\frac{\lambda}{a^*}\right)^{-\frac{1}{2}}w_{0}(\lambda^{-1}x+y_0)=\lambda^{-\frac{1}{2}}||Q(x)||_2w_{0}(\lambda^{-1}x+y_0),\quad\mathrm{for~some}\quad y_0\in\mathbb{R},
\end{equation*}
is a positive solution of the classical equation (\ref{classical eq}).

We claim that $W(x)=Q(x)$, up to a translation and a dilation.  
Note that,  via  the  translation and  dilation 
\begin{equation*}
	\begin{cases}
		w(x):=\left( \frac{\lambda}{a^*} \right)^{\frac{1}{2}} u(\lambda (x-y_0)),\\
		u(x):=\left( \frac{\lambda}{a^*} \right)^{-\frac{1}{2}} w(\lambda^{-1}x+y_0),		
	\end{cases}
\end{equation*}
the following equations are equivalent, that is,
\begin{equation*}
	\sqrt{-\Delta} w = -\lambda w + a^*w^3 \iff \sqrt{-\Delta} u + u = u^3,
\end{equation*}
which  implies that the ground states of (\ref{3.50}) and (\ref{classical eq}) correspond to each other. Once we verify that $w_0$ is a ground state of (\ref{3.50}), the claim is thus proven by applying the uniqueness of the ground state of (\ref{classical eq}). We now prove that $w_0$ is a ground state of (\ref{3.50}) by contradiction.

Let 
\begin{equation*}\label{new single}
	\tilde{I}(u):=\int_{\mathbb{R}}\left(|(-\Delta)^{\frac{1}{4}}u|^2+\lambda^{}|u|^2\right)dx-\frac{a^*}{2}\int_{\mathbb{R}}|u|^{4}dx,
\end{equation*}
and
\begin{align*}\label{new two}
	\tilde{I}_{d_1,d_2}(u_1,u_2):=&\sum_{i=1}^{2}\int_{\mathbb{R}}\left(|(-\Delta)^{\frac{1}{4}}u_i|^2+\varepsilon V_i(\varepsilon x+\varepsilon y_{\varepsilon })u_i^2-\varepsilon \mu_{d_i}|u_i|^2-\frac{d_i}{2}|u_i|^4\right)dx\notag\\	&+\frac{\beta}{2}\int_{\mathbb{R}}(|u_1|^2-|u_2|^2)^2dx, 
\end{align*}
where the Lagrange multiplier $(\mu_{d_1},\mu_{d_2})$ satisfies (\ref{3.43}).

We now suppose that there exists a solution $0\neq w_1^*\in H^{\frac{1}{2}}(\mathbb{R})$ of (\ref{3.50}) such that
\begin{equation*}
	\tilde{I}(w_1^*)<\tilde{I}(w_{0}).
\end{equation*}
Since $w_1^*$ and $w_0$ are solutions of (\ref{3.50}), we deduce that
\begin{equation}\label{5.57}
	\frac{a^{*}}{2}\int_{\mathbb{R}}|w_{1}^*|^{4}dx=\tilde{I}(w_{1}^*)<\tilde{I}(w_{0})=\frac{a^{*}}{2}\int_{\mathbb{R}}|w_{0}|^{4}dx.
\end{equation}
Given that $E_{d_{1},d_{2}}(u_{d_{1}},u_{d_{2}})=e(d_1,d_2)$, it  easily follows that
\begin{equation*}
	\tilde{I}_{d_1,d_2}(w_{d_1},w_{d_2})=\inf\left\{\tilde{I}_{d_1,d_2}(u_1,u_2):(u_1,u_2)\in\mathcal{M}\right\}.
\end{equation*}
We then have
\begin{equation*}
	\tilde{I}_{d_1,d_2}\left(\frac{w_1^*}{||w_1^*||_2},\frac{w_1^*}{||w_1^*||_2}\right)\geq\tilde{I}_{d_1,d_2}(w_{d_1},w_{d_2})=\frac{d_1}{2}\int_{\mathbb{R}}|w_{d_1}|^{4}dx+\frac{d_2}{2}\int_{\mathbb{R}}|w_{d_2}|^{4}dx-\frac{\beta}{2}\int_{\mathbb{R}}(|w_{d_1}|^2-|w_{d_2}|^2)^2dx.
\end{equation*}
This,  combined with (\ref{5.57}) and Fatou's lemma, gives that
\begin{align*}\tilde{I}(w_{0})&=\frac{a^{*}}{2}\int_{\mathbb{R}}|w_0|^{4}dx\\&\leq\operatorname*{liminf}_{(d_1,d_2)\to (a^*,a^*)}\frac{1}{2}\bigg(\frac{d_1}{2}\int_{\mathbb{R}}|w_{d_1}|^{4}dx+\frac{d_2}{2}\int_{\mathbb{R}}|w_{d_2}|^{4}dx-\frac{\beta}{2}\int_{\mathbb{R}}(|w_{d_1}|^2-|w_{d_2}|^2)^2dx\bigg)\\&=\lim_{(d_1,d_2)\to (a^*,a^*)}\frac{1}{2}\tilde{I}_{d_1,d_2}(w_{d_1},w_{d_2})\\&\leq\operatorname*{liminf}_{(d_1,d_2)\to (a^*,a^*)}\frac{1}{2}\tilde{I}_{d_1,d_2}\left(\frac{w_1^*}{||w_1^*||_2},\frac{w_1^*}{||w_1^*||_2}\right)\\&=\frac{1}{||w_{1}^*||_{2}^{2}}\int_{\mathbb{R}}\left(|(-\Delta)^{\frac{1}{4}}w_{1}^*|^{2}+\lambda |w_{1}^{*}|^2\right)dx-a^{*}\frac{1}{2\|w_{1}^*\|_{2}^{4}}\int_{\mathbb{R}}|w_{1}^*|^{4}dx\\&=\frac{a^{*}}{2}\int_{\mathbb{R}}|w_{1}^*|^{4}dx\left(\frac{2}{||w_{1}^*||_{2}^{2}}-\frac{1}{||w_{1}^*||_{2}^{4}}\right)\\&\leq\frac{a^{*}}{2}\int_{\mathbb{R}}|w_{1}^*|^{4}dx\\&=\tilde{I}(w_{1}^*).\end{align*}
This leads to a contradiction with (\ref{5.57}), and thus  $w_0$ is the unique positive least energy solution of  (\ref{3.50}),  from which we can deduce that $W(x)=Q(x)$ up to a translation and a dilation. Since $w_0>0$, one can deduce from the uniqueness (up to translations) of ground  states for the equation (\ref{classical eq}) that   $w_0$ is of the form (\ref{3.39}), which thus completes the proof of (\ref{3.39}).\\
It is easy to see that $||w_0||_2^2=1$, which,  together with the norm preservation implies that
\begin{equation*}
	w_{d_{ik}}\overset{k}{\operatorname*{\to}}w_0\quad\text{strongly in } L^2(\mathbb{R}).
\end{equation*}
In addition, this strong convergence holds also for all $p\geq2$, due to the boundedness of $\{w_{d_{ik}}\}$ in $H^{\frac{1}{2}}(\mathbb{{R}})$. Since  $w_{d_{ik}}$ and $w_0$ satisfy (\ref{3.45}) and (\ref{3.48}) respectively,   a simple analysis  yields that
\begin{equation*}
	w_{d_{ik}}\overset{k}{\operatorname*{\to}}w_0\quad\text{strongly in }H^{\frac12}(\mathbb{R}),\quad i=1,2.
\end{equation*}
Hence, (\ref{3.38}) and (\ref{3.39})  are proven.

Finally, we prove Proposition \ref{pro 3.1} concerning  the optimal estimates for $e(d_1,d_2)$ by means of the above lemmas.

\end{proof}

\begin{proof}[Proof of Proposition \ref{pro 3.1}]
	For any sequence $\{ ( d_{1k}, d_{2k}) \}$ with   $(d_{1k},d_{2k})\nearrow(a^*, a^*)$ as $k\to\infty$, from Lemma \ref{lemma 3.5} $(ii)$, we deduce that there exists a convergent subsequence, still denoted by  $\{(d_{1k},d_{2k})\}$, such that (\ref{3.38}) holds and 
	$\epsilon_ky_{\epsilon_k}\overset{k}{\operatorname*{\to}}x_0\in\Lambda$. Without loss of generality, we suppose that $x_{0}=x_{1j_{0}}$ for some $1\leq j_{0}\leq l$. Firstly, we show that
	\begin{equation}\label{cliam}
		\operatorname*{lim\sup}_{k\to\infty}\frac{|\epsilon_ky_{\epsilon_k}-x_{1j_0}|}{\epsilon_k}<\infty.
	\end{equation}
	In fact, it follows  from (\ref{3.38}) and (\ref{3.39}) that, for some $R_0>0$,
\begin{align}\label{3.59}\notag e(d_{1k},d_{2k})&=E_{d_{1k},d_{2k}}(u_{d_{1k}},u_{d_{2k}})\notag\\&\geq\sum_{i=1}^{2}\Big\{\frac{1}{\epsilon_{k}}\Big[\int_{\mathbb{R}}|(-\Delta ) ^\frac{1}{4} w_{d_{ik}}(x)|^{2}dx-\frac{a^{*}}{2}\int_{\mathbb{R}}|w_{d_{ik}}(x)|^{4}dx\Big]\notag\\&+\frac{a^{*}-d_{ik}}{2\epsilon_{k}}\int_{\mathbb{R}}|w_{d_{ik}}(x)|^{4}dx+\int_{\mathbb{R}}V_{i}(\epsilon_{k}x+\epsilon_{k}y_{\epsilon_{k}})w_{d_{ik}}^{2}dx\Big\}\\&\geq\sum_{i=1}^{2}\Big\{\frac{a^{*}-d_{ik}}{4\epsilon_{k}}\int_{\mathbb{R}}|w_{0}(x)|^{4}dx+\int_{B_{R_{0}}(0)}V_{i}(\epsilon_{k}x+\epsilon_{k}y_{\epsilon_{k}})w_{d_{ik}}^{2}dx\Big\}\notag\\&\geq C_{1}\frac{2a^{*}-d_{1k}-d_{2k}}{\epsilon_{k}}+C_{2}\sum_{i=1}^{2}\epsilon_{k}^{p_{ij_{0}}}\int_{B_{R_{0}}(0)}\left|x+\frac{\epsilon_{k}y_{\epsilon_{k}}-x_{1j_{0}}}{\epsilon_{k}}\right|^{p_{ij_{0}}}w_{d_{ik}}^{2}dx\textcolor{red}{.}\notag\end{align}
We now assume that there exists  a subsequence satisfying $\frac{|\epsilon_ky_{\epsilon_k}-x_{1j_0}|}{\epsilon_k}\to\infty\mathrm{~as~}k\to\infty$. Applying  Fatou's Lemma, it then follows from $(\ref{3.36})$ and $(\ref{3.59})$ that for any $m>0$,
\begin{align}\label{3.60}e(d_{1k},d_{2k})&\geq C_{1}\frac{2a^{*}-d_{1k}-d_{2k}}{\epsilon_{k}}+C_{2}m\epsilon_{k}^{\bar{p}_{j_{0}}}\geq Cm^{\frac{1}{\bar{p}_{j_{0}}+1}}\left(a^{*}-\frac{d_{1k}+d_{2k}}{2}\right)^{\frac{\bar{p}_{j_{0}}}{\bar{p}_{j_{0}}+1}}\notag\\&\geq Cm^{\frac{1}{\bar{p}_{j_0}+1}}\left(a^*-\frac{d_{1k}+d_{2k}}{2}\right)^{\frac{p_0}{p_0+1}},\end{align}
where $p_{0}=\operatorname*{max}_{1\leq j\leq l}\bar{p}_{j}$ and $\bar{p}_{j_{0}}=\operatorname*{min}\{p_{1j_{0}},p_{2j_{0}}\}>0$ are defined in (\ref{index}). Note that the estimate (\ref{3.60}) contradicts  the upper bound of (\ref{3.5}), and (\ref{cliam})  is thus proven.

	It can be deduced from (\ref{cliam}) that there exists a subsequence of $\{\epsilon_k\}$, still denoted by $\{\epsilon_k\}$, such that as $k \to \infty$, 
	\begin{equation*}
		\frac{\epsilon_ky_{\epsilon_k} - x_{1j_0}}{\epsilon_k} \to \bar{y}_0 \quad \text{holds for  some } \bar{y}_0\in\mathbb{{R}}.
	\end{equation*}
It then follows from (\ref{3.59}) that there exists  $C_1>0$, independent of $d_{1k}$ and $d_{2k}$, such that
\begin{equation*}
	e(d_{1k},d_{2k})\geq C_1\left(a^*-\frac{d_{1k}+d_{2k}}{2}\right)^{\frac{\bar{p}_{j_0}}{\bar{p}_{j_0}+1}}\quad\mathrm{as}\quad(d_{1k},d_{2k})\nearrow(a^*,a^*).
\end{equation*}
Based on  $\bar{p}_{j_0}\leq p_0$ and the upper bound of (\ref{3.5}), the above estimate implies that $\bar{p}_{j_{0}}=p_{0}$ and (\ref{3.3}) holds for the above subsequence $\{(d_{1k},d_{2k})\}$.

We now turn to proving that (\ref{3.4}) is satisfied  for the above subsequence $\{(d_{1k},d_{2k})\}$. We assume that 
\begin{equation*}
	\mathrm{either}\quad\epsilon_k>>\left(a^*-\frac{d_{1k}+d_{2k}}{2}\right)^{\frac{1}{p_0+1}}\mathrm{\quad or \quad}0<\epsilon_k<<\left(a^*-\frac{d_{1k}+d_{2k}}{2}\right)^{\frac{1}{p_0+1}}\mathrm{\quad as \quad }k\to\infty,
\end{equation*}
which, together with (\ref{3.59}),  implies that $e(d_{1k},d_{2k})>>\left(a^{*}-\frac{d_{1k}+d_{2k}}{2}\right)^{\frac{p_{0}}{p_{0}+1}}$ as $k\to\infty$. This leads to a contradiction with (\ref{3.3}),  and thus (\ref{3.4}) is proved.

As shown above, Proposition \ref{pro 3.1} holds for any given subsequence  $\{(d_{1k},d_{2k})\}$ satisfying $(d_{1k},d_{2k}) \nearrow (a^*,a^*)$. It thus   essentially holds for the whole sequence $\{(d_{1},d_{2})\}$ with $(d_{1},d_{2}) \nearrow (a^*,a^*)$.
\section{Proof of Theorem \ref{blow}.}\label{section new}
This section concentrates on  proving Theorem \ref{blow} related to the mass concentration of minimizers. Given the notations in (\ref{new energy}) and (\ref{new problem}), proving Theorem \ref{blow}   only necessitates  establishing the following theorem concerning the minimizers of (\ref{new problem}) as $(d_1,d_2)\nearrow(a^*,a^*)$. 
\begin{theorem}\label{Theorem 4.1}
		Let $0< \beta < a^*$, and  suppose that for $i=1,2$,    $p_{ij}\in(0,1)$  for  $1\leq j\leq n_{i}$   and  $V_i( x)$ satisfies (\ref{form 1}) - (\ref{point}). Let $( u_{d_{1k}}, u_{d_{2k}})$  be a non-negative minimizer of (\ref{new problem}) as  $(d_{1k}, d_{2k})\nearrow ( a^* , a^* )$. Then for any  sequence $\{(d_{1k}, d_{2k})\}$ satisfying  $(d_{1k}, d_{2k})\nearrow ( a^* , a^*)$ as $k\to\infty$, there exists a subsequence of  $\{(d_{1k}, d_{2k})\}$,  still denoted by $\{(d_{1k}, d_{2k})\}$, such that for
	$i = 1, 2$, each $u_{d_{ik}}$ has at least one  global maximum point $x_{ik}$ satisfying   
	\begin{equation}\label{4.1}
		x_{ik}\xrightarrow{k} \bar{x}_0\in\mathcal{Z} \quad and \quad
		\lim_{k\to\infty}\frac{|x_{ik}-\bar{x}_0|}{(a^*-\frac{d_{1k}+d_{2k}}{2})^{\frac{1}{p_0+1}}}=0.
	\end{equation}
	In addition, for $i=1,2$,
	\begin{equation*}
		\lim_{k\to\infty}\left(a^*-\frac{d_{1k}+d_{2k}}{2}\right)^{\frac{1}{2(p_0+1)}}u_{d_{ik}}\left(\left(a^*-\frac{d_{1k}+d_{2k}}{2}\right)^{\frac{1}{p_0+1}}x+x_{ik}\right)=\frac{\lambda^{\frac{1}{2}}}{\|Q\|_2}Q(\lambda x)
	\end{equation*}
	strongly in $H^{\frac{1}{2}}(\mathbb{{R}})$. Here  $\lambda$ is given by 
	\begin{equation*}
		\lambda=\left(\frac{p_0\gamma}{2}\int_{\mathbb{R}}|x|^{p_0}Q^2(x)dx\right)^{\frac{1}{p_0+1}},
	\end{equation*}
	where $p_0>0$ and $\gamma>0$ are defined in (\ref{index}) and (\ref{flattest}), respectively.
\end{theorem}

Let $(u_{d_1}, u_{d_2}) $ be a non-negative minimizer of (\ref{new problem}) with $(d_1, d_2) \nearrow (a^*, a^*)$ and denote 
\begin{equation}\label{4.3}
	\varepsilon:=(a^{*}-\frac{d_{1}+d_{2}}{2})^{\frac{1}{p_{0}+1}}>0.
\end{equation}
We then deduce from (\ref{3.4}) and (\ref{3.5}) that, for $i=1,2$,
\begin{equation*}\label{4.4}
	\sum_{i=1}^2\int_{\mathbb{R}}V_i(x)u_{d_i}^2dx\leq e(d_1,d_2)<C\left(a^*-\frac{d_1+d_2}{2}\right)^{\frac{p_0}{p_0+1}}
\end{equation*}
and 
\begin{equation}\label{4.5}
	\int_{\mathbb{R}}|(-\Delta ) ^\frac{1}{4} u_{d_i}(x)|^2dx\sim\varepsilon^{-1},\quad\int_{\mathbb{R}}|u_{d_i}(x)|^4dx\sim\varepsilon^{-1}, \quad \text{as}\quad (d_1,d_2)\nearrow(a^*,a^*). 
\end{equation}
Similar to (\ref{3.36}), it is known that there exists a sequence $\{y_{\varepsilon}\}$, together with positive constants $R_0$ and $ \eta_i$ such that 
\begin{equation}\label{4.6}
	\operatorname*{liminf}_{\varepsilon\searrow0}\int_{B_{R_0}(0)}|w_{d_i}|^2dx\geq\eta_i>0, \quad i=1,2,
\end{equation}
where $w_{d_i}$ is  the $L^2(\mathbb{R})$-normalized function given by 
\begin{equation}\label{4.7}
	w_{d_i}(x)=\varepsilon^{\frac{1}{2}} u_{d_i}(\varepsilon x+\varepsilon y_\varepsilon),\quad i=1,2.
\end{equation}
From (\ref{4.5}), we conclude that there exists a positive constant $M$,  independent of $d_1$ and $d_2$, such that
\begin{equation*}
	M\leq\int_{\mathbb{R}}|(-\Delta ) ^\frac{1}{4} w_{d_i}|^2dx\leq\frac{1}{M},\quad M\leq\int_{\mathbb{R}}|w_{d_i}|^4dx\leq\frac{1}{M}, \quad i=1,2.
\end{equation*}
\begin{proof}[Proof of Theorem \ref{Theorem 4.1}]
Define $\varepsilon _k: = \left ( a^* - \frac {d_{1k}+ d_{2k}}2\right ) ^{\frac 1{p_0+ 1}}> 0$ with $(d_{1k},d_{2k})\nearrow$ $(a^*,a^*)$ as $k\to\infty$, and let $(u_{1k}(x),u_{2k}(x)):=(u_{d_{1k}}(x),u_{d_{2k}}(x))$  be a non-negative
minimizer of (\ref{new problem}). Inspired by \cite{T18}, the proof of Theorem \ref{Theorem 4.1} is divided into  the following three steps.\\

\noindent\textbf{Step 1. Decay for estimate for $(u_{1k}(x), u_{2k}(x))$.} Let  $w_{ik}(x):=w_{d_{ik}}(x)\geq 0$ be defined by  (\ref{4.7}). By analogy with the proof of Lemma \ref{lemma 3.5} $(ii)$,  there exists a subsequence of \(\{\varepsilon_k\}\), which we still denote by \(\{\varepsilon_k\}\), such that
\begin{equation}\label{4.9}
	z_k:=\varepsilon_ky_{\varepsilon_k}\overset{k}{\operatorname*{\to}}x_0\quad\text{for some }x_0\in\Lambda,
\end{equation}
where $\Lambda$ is given by (\ref{Lamda}) and for $i=1,2$, $w_{ik}$ satisfies
\begin{equation}\label{4.10}
	\begin{cases}\sqrt{-\Delta} w_{{1k}}+\varepsilon_{k}V_{1}(\varepsilon_{k}x+z_{k})w_{1k}=\varepsilon_{k}\mu_{1k}w_{1k}+d_{1k}w_{1k}^{3}-\beta(w_{1k}^{2}-w_{2k}^{2})w_{1k}\quad\mathrm{in}\quad\mathbb{R},\\\sqrt{-\Delta} w_{{2k}} +\varepsilon_{k}V_{2}(\varepsilon_{k}x+z_{k})w_{2k}=\varepsilon_{k}\mu_{2k}w_{2k}+d_{2k}w_{2k}^{3}-\beta(w_{2k}^{2}-w_{1k}^{2})w_{2k}\quad\mathrm{in}\quad\mathbb{R}.\end{cases}
\end{equation}
Here $(\mu_{1k},\mu_{2k})$ denotes a  suitable  Lagrange multiplier satisfying
\begin{equation*}
	\mu_{ik}\sim-\varepsilon_{k}^{-1}\quad\mathrm{and}\quad\mu_{1k}/\mu_{2k}\to1\quad\mathrm{as}\quad(d_{1k},d_{2k})\nearrow(a^{*},a^{*}),\quad i=1,2.
\end{equation*}
Furthermore, $w_{ik}\xrightarrow{k}w_0$ strongly in $H^{\frac{1}{2}}(\mathbb{R})$ for some $w_0>0$, where $w_0$ satisfies 
\begin{equation*}
	\sqrt{-\Delta} w_0(x)=-\lambda w_0(x)+a^*w_0^3(x)\quad\mathrm{in}\quad\mathbb{R},
\end{equation*}
with a positive constant $\lambda$. Analogously to the proof of (\ref{3.39}), we have
\begin{equation}\label{form2}
	w_0(x)=\frac{\lambda^{\frac{1}{2}}}{\|Q\|_2}Q(\lambda(x-y_0))\quad\mathrm{for~some}\quad y_0\in\mathbb{R}.
\end{equation}
Using the decay rate of $Q$  (\ref{decay}), it can be  derived that for $i=1,2$
\begin{equation}\label{4.13}
	\int_{|x|\geq R}|w_{ik}|^2 dx\to0\quad\text{as}\quad R\to \infty\quad \text{uniformly for large } k.
\end{equation}
From (\ref{4.10}), we deduce that, for $i=1,2$,
\begin{equation}\label{weak1}
	\sqrt{-\Delta} w_{ik}(x)\leq c_{i}(x)w_{ik}(x)\quad \text{in} \quad \mathbb{{R}},
\end{equation}
where $c_i(x)=d_{ik}w_{ik}^{2}+(-1)^i\beta(w_{1k}^2-w_{2k}^2)$ in $\mathbb{{R}}$.  Applying the non-local De Giorgi-Nash-Moser theory (see  Theorem 1.1 in \cite{T16}  or Theorem 5.4 in \cite{T17}  ) to  globally non-negative weak subsolution $w_{ik}$ $(i=1,2)$ of (\ref{weak1}),  it follows that, for $i=1,2$,
\begin{equation}\label{4.15}
	\sup_{B_{1}(\xi_0)}w_{ik}\leq C\left(\int_{B_{2}(\xi_0)}|w_{ik}|^{2}dx\right)^{\frac{1}{2}},
\end{equation}
where  $C$ is a positive constant and $\xi_0$ is an arbitrary point in $\mathbb{R}$. Hence, we deduce  from (\ref{4.13}) and (\ref{4.15}) that, for $i=1,2$,
\begin{equation}\label{4.16}
	w_{ik}(x)\to0\quad\mathrm{as}\quad|x|\to\infty\quad\text{uniformly in } k.
\end{equation}
Hence, we conclude that for each $i=1,2$,  $w_{ik}(x)$ has at least one global
 maximum point.

\noindent \textbf{Step 2. The detailed concentration behavior.} For the convergent subsequence $\{w_{ik}(x)\}$ obtained in Step 1,  we let  $\hat{z}_{ik}$ be any global maximum point of $u_{ik}(x)$, where $i=1,2$. It then follows  from (\ref{4.7}) that, for $i=1,2$, $w_{ik}$ attains its global maximum at the point $\frac{\hat{z}_{ik}-z_{k}}{\varepsilon_{k}}$, where $z_k$ is given by (\ref{4.9}).  Combining (\ref{4.6}) and (\ref{4.16}), we thus deduce that, for $i=1,2$,
\begin{equation}\label{4.17}
	\operatorname*{limsup}_{k\to\infty}\frac{|\hat{z}_{ik}-z_k|}{\varepsilon_k}<\infty.
\end{equation}
We now assert that there is a sequence $\{k\}$, passing to a subsequence if necessary,  such that
\begin{equation}\label{4.18}
	\lim_{k\to\infty}\frac{|\hat{z}_{1k}-\hat{z}_{2k}|}{\varepsilon_{k}}=0.
\end{equation}
From (\ref{4.17}), we know that there exists a sequence $\{k\}$ satisfying 
\begin{equation*}
	\lim_{k\to\infty}\frac{\hat{z}_{ik}-z_k}{\varepsilon_k}=\sigma_i\quad\mathrm{for~some~}\sigma_i\in\mathbb{R},\quad i=1,2.
\end{equation*}

\noindent Let
\begin{equation}\label{4.19}
	\bar{w}_{ik}(x):=w_{ik}\left(x+\frac{\hat{z}_{ik}-z_{k}}{\varepsilon_{k}}\right)=\varepsilon_{k}^{\frac{1}{2}}u_{ik}\left(\varepsilon_{k}x+\hat{z}_{ik}\right),\quad i=1,2.
\end{equation}
According to  Step 1, where we showed that $w_{ik}\xrightarrow{k}w_0$ strongly in $ H^{\frac{1}{2}}(\mathbb{R})$ with $w_0>0$ satisfying (\ref{form2}), we have
\begin{equation*}
	\operatorname*{lim}_{k\to\infty}\bar{w}_{ik}(x)=w_{0}(x+\sigma_i)=\frac{\lambda^{\frac{1}{2}}}{\|Q\|_{2}}Q(\lambda(x+\sigma_i-y_{0}))\quad\text{strongly in }H^{\frac{1}{2}}(\mathbb{R}),\quad i=1,2.
\end{equation*}
Noting from $(\ref{4.19})$ that the origin is a critical point of $\bar{w}_{ik}$ for all $k>0$ $(i=1,2)$, we find that it is also a critical point of $w_{0}(x+\sigma_i)$. Moreover, $z_0$ is the unique critical (maximum) point of $Q(\lambda(x-z_0))$. Therefore, we deduce that $w_{0}(x+\sigma_i)=\frac{\lambda^{\frac{1}{2}}}{\|Q\|_2}Q(\lambda(x+\sigma_i-y_0))$  is spherically symmetric about the origin, which implies that $\sigma_i=y_{0}$ and 
\begin{equation}\label{4.20}
	\lim_{k\to\infty}\bar{w}_{ik}(x)=\frac{\lambda^{\frac{1}{2}}}{\|Q\|_2}Q(\lambda x):=\bar{w}_0\quad\text{strongly in }H^{\frac{1}{2}}(\mathbb{R}),\quad i=1,2.
\end{equation}
Consequently, (\ref{4.19}) and (\ref{4.20})  imply  that (\ref{4.18}). \\

\noindent\textbf{Step 3. Completion of the proof.} For $1\leq j\leq l$, denote 
\begin{equation*}
\gamma_{j}(x)=\frac{V_{1}(x)+V_{2}(x)}{|x-x_{1j}|^{\bar{p}_{j}}},
\end{equation*}
where $\bar{p}_{j}>0$ is given by (\ref{index}). It then follows from the assumptions on $V_1(x)$ and $V_2(x)$ that the limit $\lim_{x\to x_{1j}}\gamma_{j}(x)=\gamma_{j}(x_{1j})$ exists for all $j= 1,\cdots,l$. Furthermore, for $\bar{\Lambda}$ and $\gamma_{j}$ given by (\ref{1.17}) and  (\ref{gamma})  respectively, we know that $\gamma_{j}(x_{1j})=\gamma_{j}\geq\gamma$ if $x_{1j}\in\bar{\Lambda}$. 
It thus follows from (\ref{4.19}) that  
	\begin{align}\label{4.22}e(d_{1k},d_{2k})&=E_{d_{1k},d_{2k}}(u_{d_{1k}},u_{d_{2k}})\notag\\&\geq\sum_{i=1}^{2}\Big\{\frac{1}{\varepsilon_{k}}\Big[\int_{\mathbb{R}}|(-\Delta ) ^\frac{1}{4}\bar{w}_{ik}(x)|^{2}dx-\frac{a^{*}}{2}\int_{\mathbb{R}}|\bar{w}_{ik}(x)|^{4}dx\Big]\notag\\&+ \frac{a^{*}-d_{ik}}{2\varepsilon_{k}}\int_{\mathbb{R}}|\bar{w}_{ik}(x)|^{4}dx+\int_{\mathbb{R}}V_{i}(\varepsilon_{k}x+\hat{z}_{ik})\bar{w}_{ik}^{2}dx\Big\}\\&\geq\sum_{i=1}^{2}\left[\frac{a^{*}-d_{ik}}{2\varepsilon_{k}}\int_{\mathbb{R}}|\bar{w}_{ik}(x)|^{4}dx+\int_{\mathbb{R}}V_{i}(\varepsilon_{k}x+\hat{z}_{ik})\bar{w}_{ik}^{2}dx\right].\notag\end{align}
  Here $\hat{z}_{ik}$ is the   global maximum point of $u_{ik}$, with $\hat{z}_{ik}\overset{k}{\operatorname*{\to}}x_{0}\in\Lambda$ for  $i=1,2$. We may suppose that $x_{0}=x_{1j_{0}}$ for some $1\leq j_{0}\leq l$.
  
  We now claim  that, for $i=1,2$,
  \begin{equation}\label{4.23}
  	\frac{|\hat{z}_{ik}-x_{1j_0}|}{\varepsilon_k}\quad\text{is uniformly bounded as }k\to\infty.
  \end{equation}
  By contradiction, if (\ref{4.23}) is false for $i=1$ or $i=2$, we then deduce from (\ref{4.18}) that both are unbounded, which implies that there exists a subsequence of $\{(d_{1k},d_{2k})\}$, still denoted by $\{(d_{1k},d_{2k})\}$, such that
  \begin{equation*}
  	\lim_{k\to\infty}\frac{|\hat{z}_{ik}-x_{1j_0}|}{\varepsilon_k}=\infty,\quad i=1,2.
  \end{equation*}
  It then follows from  Fatou's Lemma that for any  sufficiently large $C > 0$, 
  \begin{align*}&\liminf_{k\to\infty}\varepsilon_{k}^{-\bar{p}_{j_{0}}}\sum_{i=1}^{2}\int_{\mathbb{R}}V_{i}(\varepsilon_{k}x+\hat{z}_{ik})|\bar{w}_{ik}(x)|^{2}dx\notag\\&=\liminf_{k\to\infty}\sum_{i=1}^{2}\int_{\mathbb{R}}\frac{V_{i}(\varepsilon_{k}x+\hat{z}_{ik})}{|\varepsilon_{k}x+\hat{z}_{ik}-x_{1j_{0}}|^{\bar{p}_{j_{0}}}}\Big|x+\frac{\hat{z}_{ik}-x_{1j_{0}}}{\varepsilon_{k}}\Big|^{\bar{p}_{j_{0}}}|\bar{w}_{ik}(x)|^{2}dx\\ &\geq\sum_{i=1}^{2}\int_{\mathbb{R}}\liminf_{k\to\infty}\left(\frac{V_{i}(\varepsilon_{k}x+\hat{z}_{ik})}{|\varepsilon_{k}x+\hat{z}_{ik}-x_{1j_{0}}|^{\bar{p}_{j_{0}}}}\Big|x+\frac{\hat{z}_{ik}-x_{1j_{0}}}{\varepsilon_{k}}\Big|^{\bar{p}_{j_{0}}}|\bar{w}_{ik}(x)|^{2}\right)dx\geq C,\notag\end{align*}
which, together with (\ref{4.22}), implies that
\begin{equation}\label{4.24}
	e(d_{1k},d_{2k})\geq C\varepsilon_k^{\bar{p}_{j_0}}=C\left(a^*-\frac{d_{1k}+d_{2k}}{2}\right)^{\frac{\bar{p}_{j_0}}{p_0+1}}
\end{equation}
holds for an arbitrary constant $C>0$. This leads to a contradiction with Proposition \ref{pro 3.1} since $\bar{p}_{j_0}\leq p_0$, which completes the proof of (\ref{4.23}). 

Next, we prove that $\bar{p}_{j_{0}}=p_{0}$, that is,  $x_{1j_{0}}\in\bar{\Lambda}$, where $\bar{\Lambda}$ is given by  (\ref{1.17}).  It is known from (\ref{4.23}) that there exists a subsequence of $\{(d_{1k},d_{2k})\}$ such that, for $i=1,2$,
\begin{equation}\label{4.25}
	\lim_{k\to\infty}\frac{\hat{z}_{ik}-x_{1j_0}}{\varepsilon_k}=\hat{z}_0\quad\mathrm{for~some~}\hat{z}_0\in\mathbb{R}.
\end{equation}
Noticing that $Q$ is radially decreasing and decays as $|x|\to\infty$, we then conclude from (\ref{4.20}) that  
\begin{align}\label{4.26*}&\liminf_{k\to\infty}\varepsilon_{k}^{-\bar{p}_{j_{0}}}\sum_{i=1}^{2}\int_{\mathbb{R}}V_{i}(\varepsilon_{k}x+\hat{z}_{ik})|\bar{w}_{ik}(x)|^{2}dx\notag\\=&\liminf_{k\to\infty}\sum_{i=1}^{2}\int_{\mathbb{R}}\frac{V_{i}(\varepsilon_{k}x+\hat{z}_{ik})}{|\varepsilon_{k}x+\hat{z}_{ik}-x_{1j_{0}}|^{\bar{p}_{j_{0}}}}\Big|x+\frac{\hat{z}_{ik}-x_{1j_{0}}}{\varepsilon_{k}}\Big|^{\bar{p}_{j_{0}}}|\bar{w}_{ik}(x)|^{2}dx\\ \geq&\gamma_{j_{0}}(x_{1j_{0}})\int_{\mathbb{R}}|x+\hat{z}_{0}|^{\bar{p}_{j_{0}}}\bar{w}_{0}^{2}dx\geq\frac{\gamma_{j_{0}}(x_{1j_{0}})}{\lambda^{\bar{p}_{j_{0}}}\|Q\|_{2}^{2}}\int_{\mathbb{R}}|x|^{\bar{p}_{j_{0}}}Q^{2}dx,\notag\end{align}
where equality is attained if and only if $\hat{z}_{0}=0$.  Therefore, $\bar{p}_{j_{0}}=p_{0}$ follows from (\ref{4.22}) and (\ref{4.26*}). If not, (\ref{4.24}) holds with $C$ replaced by some $\hat{C}>0$, which contradicts Proposition \ref{pro 3.1}.

Now, according to the fact that $\bar{p}_{j_0}=p_0$, we conclude that $x_{1j_0}\in\bar{\Lambda}$ and $\gamma_{j_0}(x_{1j_0})=\gamma_{j_0}$. Then, by (\ref{eq1.11}), (\ref{4.3}),  (\ref{4.22}) and  (\ref{4.26*}) , direct calculations yield
 that
\begin{align}\label{4.27}\liminf_{k\to\infty}\frac{e(d_{1k},d_{2k})}{\varepsilon_{k}^{p_{0}}}&\geq\|\bar{w}_{0}\|_{4}^{4}+\gamma_{j_{0}}\int_{\mathbb{R}}|x+\hat{z}_{0}|^{p_{0}}\bar{w}_{0}^{2}dx\notag\\&\geq\frac{1}{a^{*}}\left(2\lambda+\frac{\gamma_{j_{0}}}{\lambda^{p_{0}}}\int_{\mathbb{R}}|x|^{p_{0}}Q^{2}dx\right),\end{align}
and equality  holds  in the last inequality   if and only if $\hat{z}_0=0$. By taking the infimum  of (\ref{4.27}) over $\lambda>0$, we further obtain that
\begin{equation}\label{4.28}
	\liminf_{k\to\infty}\frac{e(d_{1k},d_{2k})}{\varepsilon_k^{p_0}}\geq\frac{2p_0+2}{p_0a^*}\left(\frac{p_0\gamma_{j_0}\int_{\mathbb{R}}|x|^{p_0}Q^2dx}{2}\right)^{\frac{1}{p_0+1}},
\end{equation}
where the equality is attained at
\begin{equation*}
	\lambda=\lambda_0:=\left(\frac{p_0\gamma_{j_0}\int_{\mathbb{R}}|x|^{p_0}Q^2dx}{2}\right)^{\frac{1}{p_0+1}}.
\end{equation*}

  Finally,  we prove that the limit of (\ref{4.28}) exists and equals the right-hand side of (\ref{4.28}). To show this, we  define
  \begin{equation*}
  	\phi_1(x)=\phi_2(x)=\frac{t^{\frac{1}{2}}}{\varepsilon_k^{\frac{1}{2}}\|Q\|_2}Q\left(\frac{t(x-x_{1j_1})}{\varepsilon_k}\right)\quad\mathrm{where}\quad x_{1j_1}\in\mathcal{Z}.
  \end{equation*}
Then, taking $(\phi_1(x),\phi_2(x)) $ as a trial function for $E_{d_{1k},d_{2k}}(\cdot,\cdot)$ and minimizing it over $t>0$, we deduce that
\begin{equation*}
	\liminf_{k\to\infty}\frac{e(d_{1k},d_{2k})}{\varepsilon_k^{p_0}}\leq\frac{2p_0+2}{p_0a^*}\left(\frac{p_0\gamma\int_{\mathbb{R}}|x|^{p_0}Q^2dx}{2}\right)^{\frac{1}{p_0+1}},
\end{equation*}
which, together with (\ref{4.28}), implies that $\gamma=\gamma_{j_0}$, i.e., $x_{1j_0}\in\mathcal{Z}$, and (\ref{4.28}) is actually an equality since $\gamma=\operatorname*{min}\left\{\gamma_{1},\textcolor{red}{\cdots},\gamma_{l}\right\}$. This indicates that $\lambda$ is unique, which is  independent of the chosen subsequence  and minimizes (\ref{4.27}), that is, $\lambda=\lambda_{0}$. Furthermore, if (\ref{4.28}) is an equality, then (\ref{4.27}) is also an equality. Therefore, by using $\hat{z}_{0}=0$ and (\ref{4.25}), we establish (\ref{4.1}).
\end{proof}
\end{proof}
\end{lemma}
\end{proof}
\end{proof}
\end{lemma}

\section*{Competing interests}
The authors declare that they have no conflicts of interest.

\section*{Authors contribution}
Each of the authors contributed to each part of this study equally, all authors read and approved the final manuscript.

\section*{Data availability statement}
Data sharing not applicable to this article as no data sets were generated or analysed during the current study.

\section*{Acknowledgments}

This work was supported by the Guangzhou University Postgraduate Innovation Ability Training Program (Project No. JCCX2025065).
Chungen Liu was supported by the National Natural Science Foundation of China (Grant No. 12171108). 
\noindent Jiabin Zuo was supported by the Guangdong Basic and Applied Basic Research
Foundation (2026A1515012273).

\medskip
\bibliographystyle{plain}


\end{document}